\documentclass[pdflatex,sn-mathphys-num]{sn-jnl}

\usepackage{graphicx}%
\usepackage{multirow}%
\usepackage{amsmath,amssymb,amsfonts}%
\usepackage{amsthm}%
\usepackage{mathrsfs}%
\usepackage[title]{appendix}%
\usepackage{xcolor}%
\usepackage{textcomp}%
\usepackage{manyfoot}%
\usepackage{booktabs}%
\usepackage{algorithm}%
\usepackage{algorithmicx}%
\usepackage{algpseudocode}%
\usepackage{listings}%

\theoremstyle{thmstyleone}%
\newtheorem{theorem}{Theorem}
\newtheorem{proposition}{Proposition}%
\newtheorem{lemma}{Lemma}

\newtheorem{Assumption}{Assumption}

\theoremstyle{thmstyletwo}%
\newtheorem{remark}{Remark}%

\theoremstyle{thmstylethree}%
\newtheorem{definition}{Definition}%

\begin{document}

\title[Article Title]{Theory and Design of Extended PID Control for Stochastic Systems with Structural
Uncertainties}


\author[1]{\fnm{Baoyou} \sur{Qu}}\email{qubaoyou@sdu.edu.cn}

\author*[2]{\fnm{Cheng} \sur{Zhao}}\email{zhaocheng@amss.ac.cn}
\equalcont{These authors contributed equally to this work.}

\affil[1]{\orgdiv{Research Center for Mathematics and Interdisciplinary Sciences}, \orgname{Shandong University}, \orgaddress{\city{Qingdao}, \postcode{266237}, \country{China}}}

\affil[2]{\orgdiv{State Key Laboratory of Mathematical
Sciences, Academy of Mathematics and Systems Science}, \orgname{Chinese Academy of
Sciences}, \orgaddress{\city{Beijing}, \postcode{100190}, \country{China}}}



\abstract{Since the classical proportional-integral-derivative (PID) controller has continued to be the most widely used feedback methods in engineering systems by far, it is crucial to investigate the working mechanism of PID in dealing with nonlinearity, uncertainty and random noises. Recently, Zhao and Guo (2022) \cite{zhao2020} has established the global stability of PID control for a class of uncertain nonlinear control systems with relative degree two without random perturbations. In this paper, we will consider a more general class of  nonlinear stochastic systems with an arbitrary relative degree $n$, and discuss the stability and design of extended PID controller (a natural extension of PID). We demonstrate that,   the closed-loop control systems will be globally stable in mean square with bounded tracking errors provided the extended PID parameters are selected  from an $(n+1)$-dimensional unbounded set, even if both the system nonlinear drift and diffusion terms contain a wide range of structural uncertainties. Moreover, the steady-state tracking error is proved to be proportional to the noise intensity at the setpoint, which can also be made arbitrarily small by choosing the controller parameters suitably large.}





\maketitle

\section{Introduction}
As the most widespread control method in the industrial automation domain, the PID controller is favored by various control practitioners because of its simple feedback structure, strong robustness and ease of implementation. In fact, PID controller is often viewed as a simplistic computational control algorithm, and is considered to be ``bread and butter'' of control engineering \cite{astrom2001}. A recent survey also shows that the PID controller has much
higher impact rating than most advanced control technologies, and will demonstrate its sustained vitality \cite{samad2017}. Ever since the emergence of PID controller in the early 20th century, tremendous efforts have been made in both academia and industry to improve the performance of PID loops, and many interesting tuning methods have emerged (see, e.g., \cite{chien1990,killingsworth2006,ziegler1993}), of which the Ziegler-Nichols rules are perhaps the most eminent. However, the de facto situation is that poorly tuned PID controlled loops are quite common in practice (see, e.g., \cite{bialkowski1993,ender1993,yu2006,odwyer2006}), and almost all the existing tuning methods are case dependent and heavily rely on experience or experiment or both. As mentioned in \cite{somefun2021}, tuning the PID algorithm for accurate and stable closed-loop control becomes a NP-Hard Problem. Moreover, there is no doubt that almost all of the practical control systems are subjected to nonlinearity, uncertainties and various disturbances, while the existing theoretical investigation on PID are mainly focused on linear systems \cite{astrom2006,wang2020,chen2022}, or affine nonlinear systems without random disturbances (see, e.g., \cite{romero2018,fliess2013,zhao2017,song2017,zhang2019,guo2020,zhao2021,guo2021}).

In recent years, some rigorous mathematical investigations on the theory and design of PID have been made for  uncertain nonlinear systems (see e.g.,\cite{zhao2017,cong2017,zhang2022pid,zhao2021,zhao2020,guo2021,zhao2023}). For instance, in \cite{zhao2017}, Zhao and Guo have shown that for a class of affine nonlinear uncertain systems with relative degree two, the three PID parameters can be designed to stabilize and regulate the control systems globally, provided that the nonlinear functions satisfy a Lipschitz condition. Moreover, necessary and sufficient conditions for the selection of PID parameters have also been provided in \cite{zhao2017}. To the best of our knowledge, the global stability of a class of PID controlled nonaffine nonlinear systems without random disturbances was established for the first time in \cite{zhao2020}. Then, the authors further extended the deterministic results to stochastic cases in \cite{zhang2022pid}. However, it is worth mentioning that only nonlinear systems with relative degree two are considered in the work of \cite{zhao2017,zhang2019,zhao2020,zhang2022pid,zhao2023}. In addition, in the existing study of PID controlled stochastic systems, the diffusion term is usually assumed to vanish at the setpoint (see \cite{cong2017,zhang2022pid}), which may be unrealistic for many practical systems.
	
It is natural to ask whether this ubiquitous controller can deal with more general nonlinear uncertain control systems.
In view of this, we will discuss the global stabilization problem for a class of nonaffine stochastic systems with a general relative degree $n$, and extend the corresponding deterministic results in \cite{zhao2020} and stochastic case in \cite{zhang2022pid} by a refined method. 
In contrast to most of the existing study of PID controlled stochastic systems, this article does not require the diffusion term to vanish at the setpoint, which has weakened the assumption used in the literature and  seems to be more reasonable for practical situations. On the other hand, for nonlinear stochastic control systems, particularly when the relative order $n$ is large, the extended PID parameters possess multiple degrees of freedom. Consequently, providing a simple analytical design formula for these parameters is typically a very challenging task.  To the best of our knowledge,  existing design methods typically rely on high-gain qualitative approaches, necessitating sufficiently large parameter values \cite{khalil2000,zhao2021}.

In this article, by employing a simple (but powerful) algebraic criteria of polynomial stability, we will establish the global stability of the closed-loop control systems and provide analytic design methods for the extended PID parameters. It is  worth to mention that the design rules given in this
article are quite simple and are not necessarily to be of high gain, which improves the existing qualitative design methods in the literature.  Additionally, both the design and analysis of MIMO nonlinear stochastic systems, as studied in this article, pose greater challenges compared to previous studies due to the system's structural uncertainties and the strong coupling between input and output. Indeed, the proofs of the main theorems reveal that constructing Lyapunov functions is far from straightforward, owing to the intricate structure of the corresponding positive definite matrices.  Our results also demonstrate  that the selection of extended PID parameters has wide flexibility and does not necessarily
require high gain, which improves the existing qualitative methods for parameter design in the literature, see e.g., \cite{khalil2000,guo2013}. Moreover, we will show that the steady-state tracking error can be made arbitrarily small by choosing the controller parameters suitably large.

The remainder of this paper is organized as follows. The problem formulation and main assumptions are provided in Section 2. The global stability together with parameter design of the extended PID controlled stochastic systems are discussed in Section 3. The tracking performances of the closed-loop control system are analyzed in Section 4.
Section 5 will provide proofs of the main results. Simulations are provided in Section 6, and some concluding remarks will be given at the end of this paper.

\section{Mathematical Formulation}
\subsection{Notations}
Denote $\mathbb{R}^{m\times n}$ as the space of $m\times n$ real matrices, $|x|$  as the Euclidean norm of a vector $x$, and $a^{\top}$ as the transpose of a vector or a matrix $a$.
The operator norm and the Hilbert-Schmidt norm of a matrix $P \in \mathbb{R}^{m\times n}$ are defined by $\|P\|=\sup_{|x|\leq 1,x\in \mathbb{R}^n}|Px|$ and $\|P\|_{{\rm HS}}=[\text{tr}(PP^{\top})]^{\frac{1}{2}}$ respectively, where ${\rm tr}(\cdot)$ represents the trace of a matrix.  For two symmetric matrices $S_1$ and $S_2$ in $\mathbb{R}^{n\times n}$, the notation $S_1>S_2$ or $S_2<S_1$ implies that  $S_1-S_2$ is a positive definite matrix; $S_1\geq S_2$ or $S_2\le S_1$ implies  that $S_1-S_2$ is a positive semi-definite matrix.
For a function $f(x_1,x_2\cdots,x_k)\in C^{1}(\mathbb{R}^{n_1}\times\mathbb{R}^{n_2}\times\cdots\times\mathbb{R}^{n_k},\mathbb{R}^{m})$, let
$\frac{\partial f}{\partial {x_i}}(x_1,\cdots,x_k)$ denote  the $m\times n_i$ Jacobian matrix of $f$ with respect to $x_i$ at the point $(x_1,\cdots,x_k)$.  For a random variable $X$, let $\mathbf E(X)$ denote its expectation.

\subsection{The Control System}
In this paper, we are concerned with the following nonaffine stochastic system with relative degree $n$ ($n\geq 1$):
\begin{equation}
	\label{n-order control system}
	\begin{cases}
		~\mathrm{d}x_1(t)\!\!&=~x_2(t)\mathrm{d}t,\\
		&~\vdots \\
		\mathrm{d}x_{n\!-\!1}(t)\!&=~x_n(t)\mathrm{d}t,\\
		~\mathrm{d}x_n(t)&=~f(x(t); u(t))\mathrm{d}t+g(x(t))\mathrm{d}B_t,\\
~y(t)\!\!&=~x_1(t),~~x_1\in\mathbb{R}^d,~u\in\mathbb{R}^d,
	\end{cases}
\end{equation}
where the state variables $x(t)=(x_1(t),\cdots,x_n(t))\in \mathbb{R}^{nd}$ are available for control design, $u(t)\in\mathbb{R}^d$ is the control input, $y(t)\in\mathbb{R}^d$ is the system output, $f(x;u): \mathbb{R}^{nd}\times \mathbb{R}^{d}\to \mathbb{R}^d$, $g: \mathbb{R}^{nd}\to \mathbb{R}^{d\times m}$ are continuous nonlinear functions and $(B_t)_{t\geq 0}$ is an $m$-dimensional standard Brownian motion.

Our control objective is  to \emph{globally stabilize} the system (\ref{n-order control system}) in the sense that
$\sup_{t\geq 0}\mathbf{E}[|x(t)|^2+|u(t)|^2]<\infty$ and, at the same time,  to make the output $y(t)$ track a desired setpoint $y^{*}\in\mathbb{R}^d$, under the condition that the nonlinear functions $f(\cdot)$ and $g(\cdot)$ contain uncertainty.

This paper is motivated by our recent theoretical investigation on the classical PID control for  system (\ref{n-order control system}) with relative degree two, see  \cite{zhao2020,zhang2022pid}. In fact, for the case $n=2$, it was shown in \cite{zhang2022pid} that the classical PID control
\begin{align*}
u(t)=K_{\!P} ~\!e(t)+K_{\!I} \int_{0}^{t}\!\! e(s)\mathrm{d}s+K_{\!D} \dot{e}(t),~e(t)=y^*-y(t)
\end{align*}
 can globally stabilize system (\ref{n-order control system}) in mean square, provided that the partial derivatives of the system nonlinear functions are bounded and the diffusion term $g$  vanishes at the setpoint. For dynamical systems with a general relative degree $n$, we will in this paper consider the extended PID controller, which is the natural extension of PID and is  defined by
\begin{equation}\label{extended pid controller}\begin{split}
u(t)\!=&k_{1} e(t)\!+\!k_{0} \int_{0}^{t}\!\! e(s)\mathrm{d}s\!+k_{2} \dot{e}(t)+\!\cdots\!+k_{n} e^{(n\!-\!1)}(t)\\
e(t)\!=&y^*-y(t)
\end{split}
\end{equation}
where  $e(t)$ is the regulation error,  $\dot{e}(t)$, $\cdots$, $e^{(n-1)}(t)$ are the derivatives of $e(t)$ up to $(n-1)^{\text{th}}$ order, and $k_0$, $\cdots$, $k_n$ are called the extended PID parameters to be designed.

It should be noted that if the relative degree of system (\ref{n-order control system}) is two (or one), the extended PID control (\ref{extended pid controller}) reduces to the well-known PID (or PI) controller.
\subsection{Assumptions}
Now, we introduce two basic assumptions for the system nonlinear functions $f$ and $g$.
\begin{Assumption}\label{A 1}
The drift term $f(x; u)$ is Lipschitz continuous with respect to $x$, uniformly in $u$, i.e.,
\begin{equation}\label{3.0}
	|f(x;u)-f(y;u)|\leq L|x-y|,\  x,y\in \mathbb{R}^{nd}, ~~u\in \mathbb{R}^d.
\end{equation}
Besides, the diffusion term $g(x)$ satisfies
\begin{equation}\label{4.0}
	\|g(x)-g(y)\|_{{\rm HS}}\leq M|x-y|, \  x,y\in \mathbb{R}^{nd},
\end{equation}
where the norm $\|\cdot\|_{{\rm HS}}$ is  defined by $\|P\|_{{\rm HS}}=[\text{\rm tr}(PP^{\top})]^{\frac{1}{2}}$.
\end{Assumption}

Condition (\ref{3.0}) on function $f(x;u)$ is a standard assumption for establishing global results,  which is also used extensively  in ordinary differential equation to ensure the global existence and uniqueness of solutions.
We next give some explanations to condition (\ref{4.0}). In most existing studies on PID control for nonlinear stochastic systems, the Brownian motion $(B_t)_{t\geq 0}$ is usually assumed to be one dimensional,  and the diffusion term $g$ is a $C^1$ function with bounded partial derivatives \cite{zhang2022pid,zhao2023}. In this paper, we consider the more general case where the Brownian motion $(B_t)_{t\geq 0}$ can be high dimensional, and the diffusion term $g$ is no longer  assumed to be a $C^1$ function. Besides, the boundedness of the partial derivatives of $g$ is replaced by the weaker condition (\ref{4.0}). It is also worth noting that the diffusion function $g$ is not required to be zero at the point $z^*:=(y^*,0,\cdots,0)\in\mathbb{R}^{nd}$, which is a more reasonable condition  and can be satisfied by more scenarios.

In order for the control input to have the ability to influence the state of the control system, the control gain matrix $\frac{\partial f}{\partial{u}}$ should not vanish. This inspires us to introduce the following assumption:
\begin{Assumption}\label{A 2}
The $d\times d$ Jacobian matrix $\frac{\partial f}{\partial{u}}$ has the following lower bound:
\begin{align}\label{343}
\frac{1}{2}\left[\frac{\partial f}{\partial u}+\Big(\frac{\partial f}{\partial u}\Big)^\top\right](x;u)\geq I_d,
  ~~(x;u)\in \mathbb{R}^{nd}\times \mathbb{R}^d,
\end{align}
where $I_d$ is the $d\times d$ identity matrix.
\end{Assumption}

By Assumption \ref{A 2} and
\cite[Proposition 4.1]{zhao2020}, one can see that for any given $y^*\in \mathbb{R}^d$, there exists a unique $u^*\in \mathbb{R}^d$ such that
\begin{equation}\label{unique u^* w.r.t y^*}
	f(y^*,0,\cdots,0;u^*)=0.
\end{equation}

\begin{remark}
We remark that if the control system (\ref{n-order control system}) is affine nonlinear and free from random disturbances, the extended PID controller (\ref{extended pid controller}) has the ability to stabilize the system globally (or semi-globally), provided that the system nonlinear functions satisfy certain growth conditions, see e.g., \cite{guo2021,zhao2020}.
\end{remark}

\section{Global Mean Square Stability}
In this section, we will demonstrate that for the nonaffine stochastic system (\ref{n-order control system}), the extended PID control (\ref{extended pid controller}) can still be designed to globally stabilize the system.
Before presenting the main results, we first introduce two definitions.

\begin{definition}
We say the stochastic system (\ref{n-order control system}) controlled by the extended PID (\ref{extended pid controller}) is globally stable in the sense of mean square, if for all initial state $x(0)\in\mathbb{R}^{nd}$, the solution of the closed-loop system satisfies
\begin{equation*}
    \sup_{t\geq 0}\mathbf{E}\bigl[|x(t)|^2+|u(t)|^2\bigr]<\infty.
\end{equation*}
\end{definition}

\begin{definition}\label{def 3.1}
For given nonnegative constants $L$ and $M$, we say that the $(n+1)$-parameters $k_0$, $k_1$, $\cdots,$ $k_n$ are admissible for the two-tuples $(L,M)$ if they are all positive, and there exists a positive definite matrix $P\in\mathbb{R}^{(n+1)\times (n+1)}$, such that
 \begin{align}\label{P}
 &P\ \!\![0,\cdots,0,1]^{\top}=~\![k_0,\cdots,k_{n-1},k_n]^{\top},\\
 &~\!\!PA+A^{\top}P + 2\bar kI_{n+1}<0,\label{P1}
\end{align}
 where $\bar k:=\sum_{i=0}^n k_iL+ k_n M^2$ and
\begin{align}\label{A}
A:=\begin{bmatrix}
0&1&0&\cdots&0\\
\vdots&\vdots&\vdots&\ddots&\vdots\\
0&0&0&\cdots&1\\
-k_0&-k_1&-k_2&\cdots&-k_n
\end{bmatrix}.\end{align}
\end{definition}
\vskip 0.2cm
For notational simplicity, let us denote
\begin{align}\label{z*}
z^*:=(y^*,0,\cdots,0)\in\mathbb{R}^{nd},
\end{align}
where $y^*\in\mathbb{R}^d$ is the setpoint, then we have the following theorem.


\begin{proposition}\label{Theorem 3.1}
Suppose  Assumptions \ref{A 1} and \ref{A 2} hold, and the $(n+1)$-parameters $k_0$, $k_1$, $\cdots,$ $k_n$ are admissible for the two-tuples $(L,M)$. Then,  the closed-loop system (\ref{n-order control system})-(\ref{extended pid controller})  will be globally stable. Moreover, there exist some positive constants $C_1$, $C_2$ and $\lambda$ depending  on $(k_0,\cdots,k_n,L,M)$ only, such that for all $t\geq 0,~x(0)\in\mathbb{R}^{nd},$
	\begin{equation}\label{tracking performance}
		\mathbf{E}\left[|x(t)\!-\!z^*|^2\right]\leq C_1\left[|x(0)-z^*|^2+|u^*|^2\right]e^{-\lambda t}+ C_2\|g(z^*)\|_{{\rm HS}}^2,
	\end{equation}
where $z^*$ and $u^*$ are defined in (\ref{z*}) and (\ref{unique u^* w.r.t y^*}) respectively.
\end{proposition}

The proof of Proposition \ref{Theorem 3.1} is given in Section \ref{sec V}, which employs  the Lyapunov method for stochastic dynamical systems.  It indicates that the extended PID control (\ref{extended pid controller}) can stabilize the nonlinear stochastic system (\ref{n-order control system}) globally, provided that the extended PID parameters are admissible for the two-tuples $(L,M)$. Therefore, it is crucial to find out a set of admissible control parameters in practical applications. To this end, we next give a general criterion to determine when $k_0$, $k_1$, $\cdots,$ $k_n$ are admissible.

\begin{proposition}\label{pro1} Let $L$ and $M$ be two given nonnegative constants. Then, the $(n+1)$-parameters $k_0$, $k_1$, $\cdots,$ $k_n$ are admissible for the two-tuples $(L,M)$, if all of them are positive and the following inequality hold:
\begin{align}\label{solution}
\min\Big\{k_0^2,~k_{i-1}^2\!-\!2k_{i-2}k_{i},~\!  2\le i\le n,~ k_n^2\!-\!k_{n-1}\Big\}>\bar k,
\end{align}
where $\bar k:=\sum_{i=0}^n k_iL+ k_n M^2$.
\end{proposition}

\begin{remark}
First, we remark that for any positive constants $L$ and $M$, inequality (\ref{solution})  always has solutions.
Indeed, let us choose $(n+1)$-parameters as follows:
\begin{equation*}
	k_0=k, \ k_i=3^{-i}k_{i-1}=3^{-\frac{i(i+1)}{2}}k, \ \ 1\leq i\leq n.
\end{equation*}
Since $3^{-n^2}k\leq k_i\leq 3^{-i}k$ for all $i=0,1,\cdots,n$, it is easy to see that
$$\bar k=\sum_{i=0}^n k_i L+ k M^2\leq \bigl(2L+M^2\bigr)k.$$
On the other hand, we know that the following inequalities hold:
\begin{align*}
&k_0^2=k^2,~~~k_n^2-k_{n-1}\geq 3^{-2n^2}k^2-k, \\
&k_i^2-2k_{i-1}k_{i+1}=k_{i-1}k_i\big(3^{-i}-2\times3^{-(i+1)}\big)\geq 3^{-3n^2}k^2,~ \ i=1,\cdots,n-1.
\end{align*}
Therefore, (\ref{solution}) will be satisfied for all large $k$.
Next, we point out that Proposition 2 provides a universal and concise design formula for the $n+1$-controller parameters, regardless of the relative degree $n$ of the system. This means that the selection of extended PID parameters has wide flexibility, since they can be arbitrarily chosen from an $(n+1)$-dimensional unbounded parameter set defined by (\ref{solution}). To the best of our knowledge, this appears to be the first such result in the literature, improving the existing related qualitative design
methods.
\end{remark}

Now, we are in position to present the main result in this paper, which follows immediately from Propositions  \ref{Theorem 3.1}-\ref{pro1}.
\begin{theorem}\label{Theorem 3.2}
Suppose  Assumptions \ref{A 1} and \ref{A 2} hold, and $k_0$, $\cdots,$ $k_n$ are all positive and satisfy (\ref{solution}).
 Then,  the closed-loop system (\ref{n-order control system})-(\ref{extended pid controller})  will be globally stable
and satisfy 
\begin{equation*}
    \mathbf{E}\left[|x(t)\!-\!z^*|^2\right]\leq C_1\left[|x(0)-z^*|^2+|u^*|^2\right]e^{-\lambda t}+ C_2\|g(z^*)\|_{{\rm HS}}^2, \ \forall t\geq 0,~x(0)\in\mathbb{R}^{nd},
\end{equation*}
for some $C_1$, $C_2$ and $\lambda$ depending  on $(k_0,\cdots,k_n,L,M)$ only.
\end{theorem}

\begin{remark}
By Theorem \ref{Theorem 3.2}, one can see that under some reasonable assumptions of $f$ and $g$, the extended PID control (\ref{extended pid controller}) can indeed globally stabilize the nonaffine stochastic uncertain system (\ref{n-order control system}).
Moreover, if the diffusion term $g$ vanishes at the point $z^*$, i.e., $g(z^*)=0$, then we know that $\mathbf{E}[|x(t)-z^*|^2]$ will converge to zero exponentially fast.
\end{remark}

\begin{remark}
It is not difficult to see that condition (\ref{343}) on $\frac{\partial f}{\partial u}$ can be weakened. To be precise, if (\ref{343}) is replaced by
\begin{align}\label{replace A2}
\frac{1}{2}\left[\frac{\partial f}{\partial u}+\Big(\frac{\partial f}{\partial u}\Big)^{\!\top}\right](x;u)\geq \underline{b} I_d,
  ~~(x;u)\in \mathbb{R}^{nd}\times \mathbb{R}^d,
\end{align}
where $\underline b>0$ is a constant, then the global stability of the closed-loop system (\ref{n-order control system})-(\ref{extended pid controller}) can still be guaranteed as long as the extended PID parameters are all positive and satisfy
\begin{align*}
\min\left\{k_0^2\underline b,~(k_{i-1}^2\!-\!2k_{i-2}k_{i})\underline b,~\!  2\le i\le n,~ k_n^2\underline b\!-\!k_{n-1}\right\}>\bar k,
\end{align*}
where $\bar k:=\sum_{i=0}^n k_iL+ k_n M^2$.
\end{remark}

\begin{remark}
Theorem 1 is also true for time-varying reference signals, provided that both the reference signal $y^*(t)$ and its time derivatives up to order $n-1$ are bounded functions over $[0,\infty)$. Besides, it can be shown that the steady-state tracking error has an upper bound proportional to the sum of the \emph{varying rate of the reference signal} and the \emph{noise intensity at the setpoint}, and  can be made arbitrarily small in at least two cases: i) the reference signal is slowly time-varying; ii) the PID controller gains are chosen sufficiently large, see \cite{zhao2023tracking}
for related discussions of a class of stochastic systems with relative degree two.
\end{remark}

In the final part of this section, we remark that if one only concerns about the stability of the closed-loop system, then the integral term $k_{0} \int_{0}^{t} e(s)\mathrm{d}s$ is \emph{not necessary} in general.  To be specific,
the following extended PD control
\begin{equation}\label{extended pd controller}
\begin{split}
&u(t)=k_{1} e(t)+k_{2} \dot{e}(t)+\cdots+k_{n} e^{(n-1)}(t),\\
&e(t)=y^*-y(t)\end{split}
\end{equation}
can also  stabilize the stochastic system (\ref{n-order control system}) globally.

\begin{theorem}\label{Theorem pd}
Consider the stochastic system (\ref{n-order control system}) controlled by extended PD control (\ref{extended pd controller}).
Suppose there exists a positive definite matrix $P_0\in\mathbb{R}^{n\times n}$, such that
 \begin{align}\label{PP}
 &P_0\ \!\![0,\cdots,0,1]^{\top}=~\![k_1,\cdots,k_n]^{\top},\\
 &P_0A_0+A_0^{\top}P_0 + 2\hat kI_{n}<0,\label{PP1}
\end{align}
 where $\hat k\in\mathbb{R}$ and  $A_0\in\mathbb{R}^{n\times n}$ are defined by
\begin{align*}
\hat k:=\sum_{i=1}^n k_iL+ k_n M^2,~A_0:=\begin{bmatrix}
0&1&\cdots&0\\
\vdots&\vdots&\ddots&\vdots\\
0&0&\cdots&1\\
-k_1&-k_2&\cdots&-k_n
\end{bmatrix},\end{align*}
 then the closed-loop system will be globally stable and satisfy
 \begin{equation*}
     \mathbf{E}\left[|x(t)\!-\!z^*|^2\right]\leq ~\!C_1\left[|x(0)-z^*|^2\right]e^{-\lambda t}+ C_2\left[\left|f(z^*;0)\right|^2+\left\|g(z^*)\right\|_{{\rm HS}}^2\right],~ \forall t\geq 0, \ x(0)\in\mathbb{R}^{nd},
 \end{equation*}
for some $C_1$, $C_2$ and $\lambda$ depending  on $(k_1,\cdots,k_n,L,M)$ only.
\end{theorem}
\begin{remark}Firstly, it can be seen from Theorem \ref{Theorem pd} that the ultimate bound of $\mathbf{E}[|x(t)-z^*|^2]$ depends not only on the value of $g(z^*)$, but also on the value of $f(z^*;0)$. However, such ultimate bound will not depend on $f(z^*;0)$, if the extended PID control (\ref{extended pid controller}) is applied (see Theorem \ref{Theorem 3.1}). This indicates that the integral term has the effect of eliminating the steady state offsets. Secondly, similar to Proposition \ref{pro1}, it can be shown that if  the $n-$parameters $k_1$, $\cdots,$ $k_n$ are positive and satisfy:
\begin{align*}
\min\left\{k_1^2,k_i^2-2k_{i-1}k_{i+1}, 2\le i\le n-1, k_n^2-k_{n-1}\right\}>\hat k,
\end{align*}
where $\hat k=\sum_{i=1}^n k_iL+ k_n M^2$, then $(\ref{PP})-(\ref{PP1})$ can be satisfied for some positive definite matrix $P_0.$
\end{remark}
\section{Parameter Design and Performance Analysis}
In Section 3, the global stability of the extended PID controlled stochastic system (\ref{n-order control system})-(\ref{extended pid controller}) is addressed.   However, some issues regarding the parameter design and tracking performance remain unclear for the case $n\geq 3$. Firstly, note that inequality (\ref{solution}) is a quadratic inequality involving $(n+1)$-variables $k_0$, $\cdots$, $k_n$, therefore it is not easy to find a set of solutions of (\ref{solution}) when the relative degree $n$ is large. Secondly, it is necessary and meaningful to obtain a more accurate upper bound on tracking error than (\ref{tracking performance}), because it may provide reliable design guidance for control practitioners. Thirdly, it is unclear whether the extended PID parameters can be selected to make the steady-state tracking error (i.e., the upper limit of the tracking error) arbitrarily small.

In this section, we will further address these fundamental problems. To this end, we first provide a design formula for the extended PID parameters. For any given $\lambda>0$, let $k_0$, $k_1$, $\cdots,$ $k_n$ be given by
\begin{equation}\label{parameters set of n-order system}
		k_0=k,~~~k_i=	\Big(\prod_{j=1}^i\beta_j\Big)k,~i=1,\cdots,n,	
\end{equation}
where $\beta_1,$ $\cdots,$ $\beta_n$ and $k$ satisfy
\begin{equation}\label{1616}
	\begin{split}
		&0<\beta_1<\left[n(\lambda+8M^2)\right]^{-1}\wedge 1 \\
        &0<\beta_i< \beta_{i-1}/n, \ \ 2\leq i \leq n  \\
        &k> \Big(\prod_{j=1}^n\beta_j\Big)^{-2} \big(1+3L+2L^2/(\lambda+8M^2)\big).
	\end{split}
\end{equation}


\begin{remark}\label{re5}
Under (\ref{parameters set of n-order system})-(\ref{1616}),	it can be proved that $k_0$, $k_1$, $\cdots,$ $k_n$ will satisfy \eqref{solution}. In fact, if we denote $$\widehat{\beta}_0:=1,~~ \widehat{\beta}_i:=\prod_{j=1}^i\beta_j, \ 0\leq i\leq n,$$ then it is not difficult to obtain
\begin{align*}
&\bar{k}=\sum_{i=0}^n\widehat{\beta}_ikL+\widehat{\beta}_nkM^2\leq \bigl(3L+\widehat{\beta}_nM^2\bigr)k,\\
&k_i^2-2k_{i-1}k_{i+1}=\widehat{\beta}_{i-1}\widehat{\beta}_i(\beta_i-2\beta_{i+1})k^2\\
&~~~~~~~~~~~~~~~~~~~~>(n-2)\widehat{\beta}_i^2k^2/n\\
&~~~~~~~~~~~~~~~~~~~~>\widehat{\beta}_n^2k^2,~1\leq i\leq n-1,\\
&k_n^2-k_{n-1}=\widehat{\beta}_n^2k^2-\widehat{\beta}_{n-1}k.
	\end{align*}
	Thus \eqref{solution} holds if
	\begin{equation}\label{0502-1}
		k>\bigl(\widehat{\beta}_{n-1}+3L+\widehat{\beta}_nM^2\bigr)\big/~\!\widehat{\beta}_n^2.
	\end{equation}
	Fortunately, \eqref{0502-1} holds since
	\begin{equation*}
		\widehat{\beta}_{n-1}+\widehat{\beta}_nM^2\leq \beta_2+\beta_1M^2<2/n<1.
	\end{equation*}
\end{remark}

Under Assumptions \ref{A 1} and \ref{A 2}, we present the main result in this section.
\begin{theorem}
	\label{Theorem uniform bounded with arbitrary small limit}
	Suppose that $(k_0,\cdots,k_n)$ are given by formulas (\ref{parameters set of n-order system})-(\ref{1616}), then the closed-loop system \eqref{n-order control system}-(\ref{extended pid controller})  will be globally stable
	and achieve the following tracking performance: for all $t\geq 0, ~x(0)\in\mathbb{R}^{nd}$,
	\begin{equation}\label{limit bound of x}
			\mathbf{E}\big[|x(t)-z^*|^2\big]\leq \frac{4n^3k_0^2}{k_n^2}\left[|x(0)-z^*|^2+|u^*|^2\right]e^{-\lambda t}
        +\frac{4n}{\lambda}\|g(z^*)\|_{{\rm HS}}^2.
	\end{equation}
 Moreover, if $\big\|\frac{\partial f}{\partial u}\big\|\leq R$ for some $R>0$, then for all $x(0)\in\mathbb{R}^{nd}$, we have
	\begin{equation}\label{limit lower bound of x}
		\liminf_{t\to \infty}\mathbf{E}\big[|x(t)-z^*|^2\big]\geq C_3\|g(z^*)\|_{{\rm HS}}^2,
	\end{equation}
where $C_3$ can be chosen as
\begin{equation*}
	C_3:=\lambda\Big[4(2+2L+M^2)\lambda+64(n+1)R^2\sum_{i=0}^nk_i^2\Big]^{-1}.
\end{equation*}
\end{theorem}

\begin{remark}Firstly, we point out that the proof of Theorem 3 hinges on the invertible coordinate transformation given in (\ref{transform from y to z}), which makes the transformed system (under the new coordinates $(z_1,z_2,\cdots,z_n)$)  a weakly dissipative system. Secondly, it is worth noting that both the upper bound and lower bound of the tracking error are established in Theorem \ref{Theorem uniform bounded with arbitrary small limit}, from which one can see that the tracking error can be made arbitrarily small by choosing the parameter $\lambda$ suitably large. Moreover, one can see that:

 \noindent (i) If $g(\cdot)$ vanishes at $z^*=(y^*,0,\cdots,0)$, then  $\mathbf{E}[|x(t)-z^*|^2]$ will converge to zero exponentially;

 \noindent (ii) If  $g(z^*)\neq 0$, $\mathbf{E}[|x(t)-z^*|^2]$ will not converge to zero.
\end{remark}

\begin{remark}
Denote \begin{align*}
\eta(t):=\mathbf{E}\big[|x(t)-z^*|^2\big]
=\mathbf{E}\big[ |e(t)|^2+ |\dot{e}(t)|^2+\cdots+ |e^{(n-1)}(t)|^2\big],
\end{align*}
then $\eta(t)$ can be regarded as a quantity which reflects the tracking performance. Combine (\ref{limit bound of x}) with (\ref{limit lower bound of x}), one can see
\begin{align*}
	C_3\|g(z^*)\|_{{\rm HS}}^2\le 	\liminf_{t\to \infty} \eta(t)\le \limsup_{t\to \infty}\eta(t)\le C_2\|g(z^*)\|_{{\rm HS}}^2,
	\end{align*}
where $C_2$ and $C_3$ are positive constants independent of the initial state and the setpoint $y^*$.
Therefore, the norm of $g(z^*)$, which reflects the intensity of the random disturbances at $z^*$, plays a crucial role in the tracking performance.
\end{remark}

\begin{remark}
	It is worth mentioning that by replacing Assumption \ref{A 2} with \eqref{replace A2}, where $\underline b>0$ is a constant, the closed-loop system \eqref{n-order control system}-(\ref{extended pid controller}) can still achieve global stability and the tracking performance \eqref{limit bound of x} as long as the last inequality in \eqref{1616} is replaced by
	\begin{equation}\label{0514}
		k> \frac{1}{\underline b}\Big(\prod_{j=1}^n\beta_j\Big)^{-2} \big(1+3L+2L^2/(\lambda+8M^2)\big).
	\end{equation}
\end{remark}
\vskip 0.1cm
To establish the lower bound (\ref{limit lower bound of x}), it is worth noting that boundedness of $\frac{\partial f}{\partial u}$ in Theorem \ref{Theorem uniform bounded with arbitrary small limit} can be weakened to the polynomial growth case.

\begin{proposition}\label{prop of lower bound}
	Suppose Assumptions \ref{A 1}-\ref{A 2} hold and the control gain matrix $\frac{\partial f}{\partial u}$ satisfies the following growth condition:
	\begin{equation}\label{polynomial growth condition}
		\Big\|\frac{\partial f}{\partial u}(x;u)\Big\|\leq R\big(1+|x|^{\gamma}+|u|^{\gamma}\big),\ (x,u)\in \mathbb{R}^{nd}\times \mathbb{R}^d,
	\end{equation}
    where $R$ and $\gamma$ are positive constants.
	If $\|g(z^*)\|_{{\rm HS}}\neq 0$ and  $(k_0,k_1,\cdots,k_n)$ are given by formulas (\ref{parameters set of n-order system})-(\ref{1616}) with $\lambda>16\gamma M^2$,
	there exists a constant $C_4$ depending only on $(k_0,\cdots,k_n,L,M,R,\gamma,n)$ such that
	\begin{equation*}
		\begin{split}
			\liminf_{t\to \infty}\mathbf{E}\big[|x(t)-z^*|^2\big]>\frac{C_4\|g(z^*)\|_{{\rm HS}}^2}{1+|y^*|^{2\gamma}+|u^*|^{2\gamma}+\|g(z^*)\|_{{\rm HS}}^{2\gamma}}.
		\end{split}
	\end{equation*}
\end{proposition}

\section{Proof of The Main Results}\label{sec V}
\begin{proof}[\textnormal{Proof of Proposition \ref{Theorem 3.1}}.]
First we denote
\begin{equation}\label{5.1}
	\begin{split}
		y_0(t)&=\int_{0}^{t}(x_1(s)-y^*)\mathrm{d}s+u^*/k_0, \\
        y_1(t)&=x_1(t)-y^*,~~
        y_i(t)=x_i(t), \ 2\!\leq i\!\leq n,~\\
        y(t)&=(y_1^{\top}(t),\cdots,y_n^{\top}(t))^{\top},
	\end{split}
\end{equation}
where $u^*$ is defined in (\ref{unique u^* w.r.t y^*}). Then system \eqref{n-order control system}-\eqref{extended pid controller} can be rewritten as (we omit the time variable $t$ for simplicity):
\begin{equation}\label{n-order system for y}
	\begin{cases}
		\mathrm{d}y_i\: =y_{i+1}\mathrm{d}t, \ \ \ \ 0\leq i\leq n-1,\\
		\mathrm{d}y_n=f(y+z^*; u)\mathrm{d}t+g(y+z^*)\mathrm{d}B_t,\\
		\ u \ \;\, =-\sum_{i=0}^nk_iy_i+u^*.
	\end{cases}
\end{equation}
Note that $$f(y+z^*; u)-f(y+z^*;u^*)=\int_0^1 \frac{\mathrm{d}}{\mathrm{d}t} f(y+z^*;u^*+t(u-u^*))\mathrm{d}t,$$
it can be deduced that
\begin{align}\label{24}
f(y+z^*; u)=f(y+z^*;u^*)+\theta(y;u)(u-u^*),
\end{align}
where $\theta(y;u)$ is a $d\times d$ matrix defined by
\begin{equation}\label{derivative matrix theta}
	\theta(y;u)=\int_0^1 \frac{\partial f}{\partial \bar u}(y+z^*;\bar u)\mathrm{d}t,~\bar u=u^*+t(u-u^*).
\end{equation}
By Assumption \ref{A 2}, we have
\begin{align}\label{17.1}
\frac{1}{2}\left(\theta(y;u)+\theta^\top(y;u)\right)\geq I_d,~~~y,u\in\mathbb{R}^d.
\end{align}
Note that $u(t)=-\sum_{i=0}^nk_iy_i(t)+u^*$, it follows that
\[
f(y(t)+z^*; u(t))=f(y(t)+z^*;u^*)-\theta(y(t);u(t))\sum_{i=0}^nk_iy_i(t).
\]

Now, suppose that $k_0$, $k_1$, $\cdots,$ $k_n$ are admissible for the two-tuples $(L,M)$. By Definition \ref{def 3.1}, there is a positive definite  matrix $P=(p_{ij})_{0\le i,j\le n}\in\mathbb{R}^{(n+1)\times (n+1)}$, such that (\ref{P})-(\ref{P1}) are satisfied.
Consider the following Lyapunov function:
$$V(Y)=\frac{1}{2}Y^{\mathsf{T}}\bar PY=\frac{1}{2}\sum_{i=0}^n\sum_{j=0}^n p_{ij}y_i^{\top}y_j, ~\bar P=P\otimes I_d,$$
where $Y=\left[y_0^{\mathsf{T}},\cdots, y_n^{\mathsf{T}}\right]^{\mathsf{T}}$, and $\otimes$ denotes the Kronecker product. Set
\begin{align*}
b(Y)=\begin{bmatrix}y_1\\\vdots\\ y_n\\
f(y+z^*; u)\end{bmatrix},~~
\sigma(Y)=\begin{bmatrix}\mathbf{0}_{d\times m}\\\vdots\\ \mathbf{0}_{d\times m}\\
g(y+z^*)\end{bmatrix},
\end{align*}
where $\mathbf{0}_{d\times m}$ denotes the $d\times m$ zero matrix, then the differential
	operator $\mathcal{L}$ acting on $V$ is given by (see Definition \ref{de1} in Appendix)
\begin{equation}\label{3.77}
	\begin{split}
		\mathcal{L}V(Y)&=\frac{\partial V}{\partial Y} b(Y)+\frac{1}{2}{\rm tr}\left\{\sigma^{\top}(Y) \frac{\partial^2V}{\partial Y}\sigma(Y)\right\}\\
        &=\sum_{i=0}^{n-1}\frac{\partial V}{\partial y_i}y_{i+1}+\frac{\partial V}{\partial y_{n}}f(y+z^*; u)+\frac{1}{2}{\rm tr}\left\{\sigma^{\top}(Y)\bar P \sigma(Y)\right\}.
	\end{split}
\end{equation}
From (\ref{P}), we know $p_{in}=k_i$ for $i=0,\cdots,n$, and thus $\frac{\partial V}{\partial y_{n}}=\sum_{i=0}^nk_iy_i^\top$.  Recall $u=-\sum_{i=0}^nk_iy_i+u^*$, we have
\begin{equation}\label{25}
	\begin{split}
		\frac{\partial V}{\partial y_{n}}f(y+z^*; u)&=\frac{\partial V}{\partial y_{n}}f\Big(y+z^*; -\sum_{i=0}^nk_iy_i+u^*\Big)\\
        &=\sum_{i=0}^nk_iy_i^\top\Big[f(y+z^*; u^*)-\theta(y;u)\sum_{i=0}^nk_iy_i\Big]\\
        &\leq  \sum_{i=0}^nk_iy_i^\top f(y+z^*; u^*)- \Big(\sum_{i=0}^nk_iy_i^\top\Big)\sum_{i=0}^nk_iy_i\\
        &\leq \sum_{i=0}^nk_iL|Y|^2-\frac{\partial V}{\partial y_{n}}\sum_{i=0}^nk_iy_i,
	\end{split}
\end{equation}
where the first inequality follows from the equality
\begin{align*}
\Big(\sum_{i=0}^nk_iy_i^\top\Big) \theta(y;u)\sum_{i=0}^nk_iy_i&=\Big(\sum_{i=0}^nk_iy_i^\top\Big) \theta^\top(y;u)\sum_{i=0}^nk_iy_i\\
&=\frac{1}{2}\Big(\sum_{i=0}^nk_iy_i^\top\Big) (\theta+\theta^\top)(y;u)\sum_{i=0}^nk_iy_i
\end{align*}
and property (\ref{17.1}), and the second inequality follows from Lipschitz property of $f$ (see Assumption 1) and the fact $f(z^*;u^*)=0$.

Besides, the upper bound of $g(y+z^*)$ can be estimated as follows:
\begin{align*}
\|g(y+z^*)\|_{{\rm HS}}^2\le \|g(y+z^*)-g(z^*)+g(z^*)\|_{{\rm HS}}^2
\le 2M^2|y|^2+2\|g(z^*)\|_{{\rm HS}}^2,
\end{align*}
thus
\begin{equation}\label{26}
	\begin{split}
		\frac{1}{2}{\rm tr}\Big\{\sigma(Y)^{\top}\bar P \sigma(Y)\Big\}
&=\frac{1}{2}k_n{\rm tr}\left[g^{\top}g(y+z^*)\right]
        =\frac{1}{2}k_n\|g(y+z^*)\|_{{\rm HS}}^2\\
        &\le k_nM^2|y|^2+k_n\|g(z^*)\|_{{\rm HS}}^2.
	\end{split}
\end{equation}
From (\ref{3.77})-(\ref{26}) and recall $\bar k=\sum_{i=0}^n k_iL+ k_n M^2$, we conclude that
\begin{align*}
\mathcal{L}V&\le \sum_{i=0}^{n-1}\frac{\partial V}{\partial y_i}y_{i+1}-
\frac{\partial V}{\partial y_{n}}\sum_{i=0}^nk_iy_i+\bar k|Y|^2+k_n\|g(z^*)\|_{{\rm HS}}^2\\
&= \frac{1}{2}Y^{\top} \left[\left(PA+A^{\top}P\right)\otimes I_d\right] Y+\bar k|Y|^2+k_n\|g(z^*)\|_{{\rm HS}}^2.
\end{align*}
From (\ref{P1}) and the positive definiteness of $P$, it can be seen that
$$\mathcal{L}V(Y)\le -\eta V(Y)+k_n\|g(z^*)\|_{{\rm HS}}^2$$
for some $\eta>0.$ Hence, it follows from Lemma \ref{lemma upper bound} that
\begin{equation*}
		\mathbf{E}V(Y(t))\leq V(Y(0))e^{-\eta t}+ k_n\|g(z^*)\|_{{\rm HS}}^2/\eta,~ \forall t\geq 0,
\end{equation*}
Hence, $\mathbf{E}V(Y(t))$ is a bounded function of $t\in[0,\infty)$.
Note that 
\[
|x(t)-z^*|\le |Y(t)|, \ u(t)=-\sum_{i=0}^nk_iy_i(t)+u^*, \ \text{ and } \ \lambda_{\min}(P) |Y|^2\le V(Y)\le \lambda_{\max}(P)|Y|^2,
\]
we know that the closed-loop system is globally stable in the sense that
$$\sup_{t\geq 0}\mathbf{E}\bigl[|x(t)|^2+|u(t)|^2\bigr]<\infty.$$
In addition, by applying the fact $|Y(0)|^2=|x(0)-z^*|^2+|u^*|^2/k_0^2,$
we conclude that (\ref{tracking performance}) holds for some positive constants $C_1$, $C_2$ and $\lambda$ depending  on $(k_0,\cdots,k_n,L,M)$ only.
\end{proof}

\begin{proof}[\textnormal{Proof of Proposition \ref{pro1}}.]
First, we point out that the construction of the matrix $P$ plays a key role in the proof. In fact, the matrix $P$ is constructed so that (\ref{P}) is satisfied and $PA+A^{\top}P$ is a \emph{diagonal matrix}. This requirement will uniquely determine the elements in the matrix $P$.

Now, we prove Proposition \ref{pro1} by considering four cases, $n=1$,  $n=2$,  $n=3$ and $n>3$. It should be noted that the first two situations $n\le 2$ have been considered in \cite{zhao2020}, and  we directly borrow the construction of matrix the construction of the matrix $P$ .

Case 1: When $n=1$, the extended PID control (\ref{extended pid controller}) reduces to the following PI control:
\begin{equation*}\begin{split}
u(t)=k_{1} e(t)+k_{0} \int_{0}^{t}e(s)\mathrm{d}s, ~~e(t)=y^*-y(t).
\end{split}
\end{equation*}
Besides, matrix $A$ defined in (\ref{A}) has the form $$A=\begin{bmatrix}0&1\\-k_0&-k_1\end{bmatrix}.$$
By Definition \ref{def 3.1}, it suffices to construct a $2\times 2$ matrix $P>0$ such that $$P[0,1]^\top=[k_0,k_1]^\top,~~PA+A^{\top}P + 2\bar kI_{2}<0$$ where $I_2$ is the $2\times 2$ unit matrix.

To this end, let us choose
\begin{align}\label{p0}
P=\begin{bmatrix}2k_0k_1&k_0\\
k_0&k_1\end{bmatrix},
\end{align}
then it is easy to see that
$$PA+A^\top P=
\begin{bmatrix}
-2k_0^2&0\\0&-2k_1^2
\end{bmatrix}.$$
By condition (\ref{solution}), we know that $\min\left\{k_0^2,~ k_1^2-k_{0}\right\}>\bar k>0$. The positive definiteness of the matrix $P$ defined by (\ref{p0}) can be easily verified, since $k_0$, $k_1$ are positive and $k_1^2-k_0>0$. Moreover, from $\min\left\{k_0^2,~ k_1^2-k_{0}\right\}>\bar k$, it can be deduced that $PA+A^{\top}P + 2\bar kI_{2}<0$.

Case 2: When $n=2$, the extended PID control (\ref{extended pid controller}) reduces to the classical PID control:
\begin{equation*}\begin{split}
u(t)\!=k_{1} e(t)\!+\!k_{0} \int_{0}^{t}\!\! e(s)\mathrm{d}s+ k_2\dot e(t), ~~e(t)\!=y^*-y(t).
\end{split}
\end{equation*}
Besides, matrix $A$ defined in (\ref{A}) turns into
 $$A=\begin{bmatrix}0&1&0\\0&0&1\\-k_0&-k_1&-k_2\end{bmatrix}.$$
We mention that if a $3\times 3$ symmetric matrix $P$ satisfies $P[0,0,1]^\top =[k_0,k_1,k_2]^\top$ and makes $PA+A^{\top}P$ is a \emph{diagonal matrix}, then $P$ must have the following form:
\begin{align}\label{354}
P=\begin{bmatrix}2k_0k_1&2k_0k_2&k_0\\
2k_0k_2&2k_1k_2-k_0&k_1\\
k_0&k_1&k_2\end{bmatrix}.
\end{align}
In fact, it can be calculated that
\begin{align}\label{355}
PA+A^\top P=-2\text{diag}(k_0^2,k_1^2-2k_0k_2,k_2^2-k_1).
\end{align}
By  (\ref{solution}), we have \begin{align}\label{356}
\min\Big\{k_0^2,~k_{1}^2-2k_{0}k_{2},~ k_2^2\!-\!k_{1}\Big\}>\bar k>0
\end{align}
where $\bar k:=(k_0+k_1+k_2) L+ k_2 M^2$. Obviously, it follows from (\ref{355}) and (\ref{356}) that
$PA+A^{\top}P + 2\bar kI_{3}<0$.
Next, we aim to prove the positiveness of $P$, which can be accomplished by verifying the following three inequalities:
$$2k_0k_1>0,~~\det \begin{bmatrix}
2k_0k_1&2k_0k_2\\
2k_0k_2&2k_1k_2-k_0
\end{bmatrix}>0,~~\det (P)>0.$$
The first inequality is obvious since $k_0$, $k_1$ and $k_2$ are all positive.
Next, note that $k_1^2>2k_0k_2$ and $k_2^2>k_1$, we know that
\begin{align*}
\det \begin{bmatrix}
2k_0k_1&2k_0k_2\\
2k_0k_2&2k_1k_2-k_0
\end{bmatrix}
&=2k_0k_1(2k_1k_2-k_0)-(2k_0k_2)^2\\
&=4k_0k_2(k_1^2-k_0k_2)-2k_0^2k_1\\
&> 4k_0^2k_2^2-2k_0^2k_1\\
&=2k_0^2\left(2k_2^2-k_1\right)
>0.
\end{align*}
Finally, by some simple calculations, the determinant the matrix of $P$ (see (\ref{354})) is given by
\begin{align}
\det (P)=k_0\left(4k_1^2k_2^2+k_0^2-2k_1^3-4k_0k_2^3\right),
\end{align}
which is positive since
\begin{align*}
4k_1^2k_2^2+k_0^2-2k_1^3-4k_0k_2^3
&=k_1^2(4k_2^2-2k_1)-4k_0k_2^3+k_0^2\\
&>k_1^2(4k_2^2-2k_2^2)-4k_0k_2^3+k_0^2\\
&=2k_2^2(k_1^2-2k_0k_2)+k_0^3
>k_0^3>0.
\end{align*}

Case 3: When $n=3$, it suffices to construct a $4\times 4$ matrix $P>0$ such that (\ref{P}) and (\ref{P1}) hold. Let us choose $$P:=\begin{bmatrix}
2k_0k_1&2k_0k_2&2k_0k_3&k_0\\
2k_0k_2&2k_1k_2-2k_0k_3&2k_1k_3-k_0&k_1\\
2k_0k_3&2k_1k_3-k_0&2k_2k_3-k_1&k_2\\
k_0&k_1&k_2&k_3\end{bmatrix},$$
and  $Q:=-(PA+A^{\top}P)$, where \begin{align}\label{a0}A:=\begin{bmatrix}
0&1&0&0\\
0&0&1&0\\
0&0&0&1\\
-k_0&-k_1&-k_2&-k_3
\end{bmatrix}.\end{align}
By some simple calculations, we can obtain
$$Q=2\text{diag}\left(k_0^2, k_1^2-2k_0k_2,  k_2^2-2k_1k_3+k_0, k_3^2-k_2\right).$$
Now, suppose that $(k_0,k_1,k_2,k_3)$ satisfy (\ref{solution}), i.e.,
$$\min\left\{k_0^2,~k_1^2-2k_{0}k_{2},~k_2^2-2k_1k_3,~ k_3^2-k_{2}\right\}>\bar k,$$
where $\bar k=\sum_{i=0}^3 k_iL+ k_3 M^2$, then we easily obtain
 $Q>0$.

Next, we will show that the matrix $A$ defined by (\ref{a0}) is Hurwitz.
The characteristic polynomial of $A$ is given by $s^4+k_3s^3+k_2s^2+k_1s+k_0$. By the Routh-Hurwitz criteria, $A$ is Hurwitz if and only if all $k_i>0$ and
\begin{align}\label{hurwitz}
k_1k_2k_3-k_1^2-k_0k_3^2>0.
\end{align}
 To prove (\ref{hurwitz}), we define a function $h$ as follows:
\begin{align}
h(x)=k_1k_2x-k_1^2-k_0x^2,~\sqrt{k_2}<x<\frac{k_2^2}{2k_1},
\end{align}
then the derivative of $h$ satisfies
$$h'(x)=k_1k_2-2k_0x>h'\left(\frac{k_2^2}{2k_1}\right)=\frac{(k_1^2-k_0k_2)k_2}{k_1}>0,$$
which implies
$$h(x)>h(\sqrt{k_2})=k_1k_2\sqrt{k_2}-k_1^2-k_0k_2.$$
Note that $k_2^2>2k_1k_3$ and $k_3>\sqrt{k_2}$, we know $k_2\sqrt{k_2}>2k_1$, therefore
$h(x)>k_1k_2\sqrt{k_2}-k_1^2-k_0k_2>k_1^2-k_0k_2>0.$
Thus, $A$ is Hurwitz. Moreover, note that $P$ satisfies the matrix equation
\begin{equation}\label{RHS}
XA+A^{\top}X
=-2\text{diag}(k_0^2, k_1^2-2k_0k_2,  k_2^2-2k_1k_3+k_0, k_3^2-k_2).
\end{equation}
Since $A$ is Hurwitz (and thus $A$ and $-A^{\top}$ have no common eigenvalues) and the RHS of (\ref{RHS}) is a negative definite matrix, the above matrix equation exists a unique solution which is positive definite. Hence, $P>0$. Therefore, the  parameters $k_0$, $k_1$, $k_2$ and $k_3$ are admissible for the two-tuples $(L,M)$.

Case 4: $n>3$. We prove this case in four steps.

\emph{Step 1.} (Construction of $P$) We first define a set of numbers $\{p_{0j}, 0\le j\le n\}$ as follows:
\begin{align}\label{9}
p_{0j}=2k_0k_{j+1},~0\le j\le n-1;~~p_{0n}=k_0.
\end{align}
With the help of $p_{02}, \cdots, p_{0,n}$, we further define
\begin{align*}
p_{1j}=2k_1k_{j+1}-p_{0,j+1},~1\le j\le n-1;~p_{1n}=k_1.
\end{align*}
Similarly, suppose that for some $1\le i\le n$, $p_{i-1,j}$ has been defined for $i-1\le j\le n$, we define
\begin{align}\label{12}
&p_{ij}=2k_{i}k_{j+1}-p_{i-1,j+1}, \ \ i\le j\le n-1;\\
&p_{in}=k_{i}\label{122}
\end{align}
recursively. Combine (\ref{9})-(\ref{122}), one can see that $p_{ij}$ has been well defined for all $0\le i\le j\le n$. Besides, denote
$$p_{ij}:=p_{ji}, ~\text{for}~ 0\le j<i\le n,$$
and let $P$ be an $(n+1)\times (n+1)$ matrix given by $$P:=(p_{ij})_{0\le i,j\le n},$$
where $p_{ij}$ are the numbers just defined above.

We will show that the above defined matrix $P$ is positive definite and satisfies conditions (\ref{P}) and $(\ref{P1})$.

From (\ref{122}), it is easy to see
\begin{align*}
P[0,\cdots,0,1]^{\top}=[p_{0n},p_{1n},\cdots,p_{nn}]^{\top}
=[k_0,k_1,\cdots,k_n]^{\top},
\end{align*}
which implies that (\ref{P}) holds.

\emph{Step 2.} The above defined $p_{ij}$ satisfy the following:
\begin{align}\label{14}
0\le p_{ij}\le 2k_{i}k_{j+1},~  0\le i\le j\le n-1.
\end{align}
We use induction to prove (\ref{14}).
Indeed, for $i=0$, from (\ref{9}), we know
$$0\le p_{0j}=2k_{0}k_{j+1},~  0\le j\le n-1,$$ therefore
(\ref{14}) holds for $i=0.$ Now, suppose that
(\ref{14}) holds for some $0\le i_0<n-1$. Then, by the relationship (\ref{12}), we know that
\begin{align}\label{15}
p_{i_0+1,j}=2k_{i_0+1}k_{j+1}-p_{i_0,j+1},~i_0+1\le j\le n-1,
\end{align}
which yields that $p_{i_0+1,j}\le 2k_{i_0+1}k_{j+1}$, for $i_0+1\le j\le n-1$.
Next, we aim to show that $p_{i_0+1,j}\geq 0$, for $i_0+1\le j\le n-1$.

By induction we have
\begin{align}\label{16}
	p_{i_0,j+1}\le 2k_{i_0}k_{j+2},~ i_0+1\le j\le n-2.
\end{align}
Moreover, since $k_j^2-2k_{j-1}k_{j+1}>0, 1\le j\le n-1$, we have
\begin{equation*}
	\frac{k_{i}}{k_{i-1}}>2\frac{k_{i+1}}{k_{i}}>\cdots>2^{j-i+1}\frac{k_{j+1}}{k_j}, \ 1\leq i\leq j\leq n-1.
\end{equation*}
Hence
\begin{equation}\label{0507-1}
	k_ik_j>2^{j-i+1}k_{i-1}k_{j+1}>k_{i-1}k_{j+1}, \ 1\leq i\leq j\leq n-1.
\end{equation}
Combine this with (\ref{16}) and (\ref{15}), we have $p_{i_0+1,j}\geq 0,~i_0+1\le j\le n-2.$ Finally, for $j=n-1$,
it follows from (\ref{15}) and (\ref{122}) that
$$p_{i_0+1,n-1}=2k_{i_0+1}k_n-p_{i_0,n}=2k_{i_0+1}k_n-k_{i_0}.$$ Since $k_n^2>k_{n-1}$, by \eqref{0507-1} we can obtain
\begin{align}\label{ki}
k_{i_0+1}k_n=k_{i_0+1}k_{n-1}\frac{k_n}{k_{n-1}}>2^{n-i_0-1}k_{i_0}\frac{k_n^2}{k_{n-1}}>k_{i_0}.
\end{align}
Therefore, $p_{i_0+1,n-1}>0$.   Hence, $p_{i_0+1,j}\geq 0$, for $i_0+1\le j\le n-1$ and (\ref{14}) is proved.

\emph{Step 3.} In this step, we  calculate each element of $Q:=-(PA+A^{\top}P)$, where $A$ is defined in (\ref{A}).

Denote $Q:=(q_{ij})_{0\le i,j\le n}$,
then $q_{ij}$, $i<j$ can be obtained as follows:
$$q_{ij}=
\begin{cases}
2k_{0}k_{j}-p_{0,j-1}, & 0=i<j\le n,\\
2k_{i}k_{j}-p_{i,j-1}-p_{j,i-1}, & 1\le i<j\le n.
\end{cases}$$
For $i=0$, it can be seen from (\ref{9}) that
$$q_{0j}=2k_{0}k_{j}-p_{0,j-1}=0, ~\text{for all}~ 1\le j\le n.$$
For $i\geq 1$, note that $p_{j,i-1}=p_{i-1,j}$, we conclude from \eqref{12} that
$q_{ij}=0$ for any $i<j$. Therefore, $Q$ is a diagonal matrix. Next,  the diagonal element of $Q$ can be calculated as follows:
\begin{align}q_{00}=2k_0^2,~~q_{ii}=2(k_{i}^2-p_{i-1,i}),~~1 \le i\le n,\end{align}
which in turn gives
\begin{align}\label{19}
\lambda_{\min}[Q/2]=\min\{k_0^2,k_{i}^2-p_{i-1,i},~1 \le i\le n\}.
\end{align}
In addition, from (\ref{14}), we know that $p_{i-1,i}\le 2k_{i-1}k_{i+1}$ for $1 \le i\le n-1$. Combine this with (\ref{19}), we obtain
\begin{align*}
\lambda_{\min}[Q/2]
\geq \min\{k_0^2,k_{i}^2-2k_{i-1}k_{i+1},~1 \le i\le n-1, k_n^2-k_{n-1}\}.\end{align*}
Recall $k_0$, $k_1$, $\cdots,$ $k_n$ satisfy (\ref{solution}), we have
$\lambda_{\min}[Q/2]>\bar k>0$, which in turn gives
$$PA+A^{\top}P + 2\bar kI_{n+1}<0.$$
Therefore, (\ref{P1}) is proved.

\emph{Step 4.} We show that the matrix $P$ is positive definite.
By Lemma \ref{lemma hurwitz} in the Appendix, we know that the matrix $A$ is Hurwitz.
Note that $P$ satisfies the matrix equation
$$PA+A^{\top}P=-Q.$$
Since $A$ is Hurwitz and $Q$ is positive definite, we conclude that the above matrix equation exists a unique positive definite solution,
which yields the positive definiteness of $P$. Therefore, the parameters $k_0$, $k_1$, $k_2$ and $k_3$ are admissible for the two-tuples $(L,M)$.
\end{proof}

\begin{proof}[\textnormal{Proof of Theorem \ref{Theorem pd}}.]
We only provide an outline of the proof. Denote
\begin{align}\label{5.12}\nonumber
 y_1(t)=x_1(t)-y^*, ~y_i(t)=x_i(t), \ 2\leq i\leq n,
\end{align}
then the extended PD control $u(t)=-\sum_{i=1}^n k_iy_i(t)$ and system \eqref{n-order control system} with \eqref{extended pd controller} can be rewritten as the following:
\begin{equation}
	\begin{split}
		\mathrm{d}y_i&=y_{i+1}\mathrm{d}t, \ \ \ \ \ \   1\leq i\leq n-1,\\
	    \mathrm{d}y_n&=f\Big(y+z^*; -\sum_{i=1}^n k_iy_i \Big)\mathrm{d}t+g(y+z^*)\mathrm{d}B_t,
	\end{split}
\end{equation}
where $y=\left[y_1^{\mathsf{T}},\cdots, y_n^{\mathsf{T}}\right]^{\mathsf{T}}$.
Similar to (\ref{24}), we can express $f(y+z^*; u)$ as follows:
\begin{align}
f(y+z^*; u)=f(y+z^*;0)+\theta(y;u)u,
\end{align}
where $$\theta(y;u)=\int_0^1 \frac{\partial f}{\partial \bar u}(y+z^*;tu)\mathrm{d}t$$ is a $d\times d$ matrix satisfying $\theta(y;u)\geq I_d$  for all $y,u\in\mathbb{R}^d$.

Now, consider the following Lyapunov function:
$$V(y)=\frac{1}{2}y^{\mathsf{T}}\hat P y,~~\hat P=P_0\otimes I_d,~~ y=\left[y_1^{\mathsf{T}},\cdots, y_n^{\mathsf{T}}\right]^{\mathsf{T}}$$
and set
\begin{align*}
b(Y)=\begin{bmatrix}y_2\\\vdots\\ y_n\\
f(y+z^*; u)\end{bmatrix},~~
\sigma(Y)=\begin{bmatrix}\mathbf{0}_{d\times m}\\\vdots\\ \mathbf{0}_{d\times m}\\
g(y+z^*)\end{bmatrix},
\end{align*}
 then by the definition \eqref{definition of operator L} of $\mathcal{L}$, we have
\begin{align}\label{3.7}
&\mathcal{L}V
=\sum_{i=1}^{n-1}\frac{\partial V}{\partial y_i}y_{i+1}+
\frac{\partial V}{\partial y_{n}}f(y\!+\!z^*; u)+\frac{1}{2}{\rm tr}\Big\{\sigma^{\top}(y)\hat P \sigma(y)\Big\}.
\end{align}
Note that $\frac{\partial V}{\partial y_{n}}=\sum_{i=1}^nk_iy_i^\top$ (which can be obtained from (\ref{PP})) and $|f(y+z^*; 0)|\le L|y|+|f(z^*;0)|$, it can be deduced that
\begin{align*}
\frac{\partial V}{\partial y_{n}}f(y+z^*; u)
&=\sum_{i=1}^nk_iy_i^\top\Big[f(y+z^*; 0)-\theta(y;u)\sum_{i=1}^nk_iy_i\Big]\\
&\le  \sum_{i=1}^nk_iy_i^\top f(y+z^*; 0)-\Big(\sum_{i=1}^nk_iy_i^\top\Big)\sum_{i=1}^nk_iy_i\\
&\le  \sum_{i=1}^nk_i\left(L|y|^2+|f(z^*;0)||y|\right)-\frac{\partial V}{\partial y_{n}}\sum_{i=1}^nk_iy_i.
\end{align*}
Similar to the proof of Theorem \ref{Theorem 3.1}, we know
\begin{align*}
\frac{1}{2}{\rm tr}\Big\{\sigma^{\top}(y)\hat P \sigma(y)\Big\}
\le k_nM^2|y|^2+k_n\|g(z^*)\|_{{\rm HS}}^2.
\end{align*}
From (\ref{3.7}) and recall $\hat k=\sum_{i=1}^n k_iL+ k_n M^2$, we conclude that
\begin{align*}
\mathcal{L}V&\le \sum_{i=1}^{n-1}\frac{\partial V}{\partial y_i}y_{i+1}-
\frac{\partial V}{\partial y_{n}}\sum_{i=1}^nk_iy_i+\hat k|y|^2
+\sum_{i=1}^nk_i |y||f(z^*,0)|+k_n\|g(z^*)\|_{{\rm HS}}^2\\
&= \frac{1}{2}y^{\top} \left[\left(P_0A_0+A_0^{\top}P_0\right)\otimes I_d\right]y+\hat k|y|^2
+\sum_{i=1}^nk_i |y||f(z^*,0)|+k_n\|g(z^*)\|_{{\rm HS}}^2\\
&\le -\eta |y|^2+\sum_{i=1}^nk_i |y||f(z^*,0)|+k_n\|g(z^*)\|_{{\rm HS}}^2
\end{align*}
for some $\eta>0.$
Next, note that $$\sum_{i=1}^nk_i |y||f(z^*,0)|\le \eta |y|^2/2+\Big(\sum_{i=1}^nk_i\Big)^2|f(z^*,0)|^2/(2\eta),$$
one can get
\begin{align*}
\mathcal{L}V
\le &-\eta_0 V+c_0\big(|f(z^*,0)|^2+\|g(z^*)\|_{{\rm HS}}^2\big),
\end{align*}
for some positive constants $\eta_0,c_0$ which depending only on $(k_0,\cdots,k_n,L,M)$.
 Hence, it follows from Lemma \ref{lemma upper bound} and the positive definiteness of $P_0\otimes I_d$ that
\begin{equation*}
		\mathbf{E}|y(t)|^2\leq C_1|y(0)|^2e^{-\lambda t}+  C_2\left[\left|f(z^*;0)\right|^2+\left\|g(z^*)\right\|_{{\rm HS}}^2\right].
	\end{equation*}
Recall $x(t)-z^*=y(t)$, thus the proof of Theorem \ref{Theorem pd} is complete.
\end{proof}

\begin{proof}[\textnormal{Proof of Theorem \ref{Theorem uniform bounded with arbitrary small limit}}.]
Under notations (\ref{5.1}), we first introduce an invertible coordinate transformation of the state variables.
Let us consider the following transformation:
\begin{equation}
	\label{transform from y to z}
	 z_0(t)=y_0(t), \ z_{i}(t)=z_{i-1}(t)+\biggl(\prod_{j=1}^i\beta_j\biggr)y_i(t), \ 1\leq i\leq n,
\end{equation}
where $\beta_1,\cdots,\beta_n$ are given as in \eqref{1616}. Then the following system is derived from system \eqref{n-order system for y} (or system \eqref{n-order control system}) under the new coordinates $z=(z_0,\cdots,z_n)$.
\begin{equation}
	\label{n-order system for z}
	\begin{cases}
		\mathrm{d}z_i(t)\:=\sum_{j=0}^i\frac{1}{\beta_{j+1}}(z_{j+1}(t)-z_j(t))\mathrm{d}t, \ \ \ \ \ \ \ \  0\leq i\leq n-1,\\
		\mathrm{d}z_n(t)=\sum_{j=0}^{n-1}\frac{1}{\beta_{j+1}}(z_{j+1}(t)-z_j(t))\mathrm{d}t +f_{\beta}(z(t); u(t))\mathrm{d}t+g_{\beta}(z(t))\mathrm{d}B_t,
	\end{cases}
\end{equation}
where
\begin{equation*}
    f_{\beta}(z;u)=\widehat{\beta}_nf\left(z^{\beta}+z^*;u\right), \  \ \
	g_{\beta}(z)=\widehat{\beta}_ng\left(z^{\beta}+z^*\right),
\end{equation*}
in which $\widehat{\beta}_i:=\prod_{j=1}^i\beta_j, \ 1\leq i\leq n$, and
\begin{equation*}
	  z^{\beta}:=\biggl[\frac{1}{\widehat{\beta}_1}\big(z_1^{\top}-z_0^{\top}\big), \frac{1}{\widehat{\beta}_2}\big(z_2^{\top}-z_1^{\top}\big),\cdots, \frac{1}{\widehat{\beta}_n}\big(z_n^{\top}-z_{n-1}^{\top}\big)\biggr]^{\top}.
\end{equation*}
Since $(k_0,k_1,\cdots,k_n)$ are given as in \eqref{parameters set of n-order system}, the extended PID controller can be expressed in the new coordinates \eqref{transform from y to z} as follows:
\begin{equation}\label{controller in z}
    u(t)=-k_0y_0(t)-\sum_{i=1}^nk_iy_i(t)+u^*
        =-kz_n(t)+u^*.
\end{equation}
Now set

\begin{align*}
	b(z)=\begin{bmatrix}\frac{1}{\beta_1}(z_1-z_0)\\
		\vdots\\
		\sum_{j=0}^{n-1}\frac{1}{\beta_{j+1}}(z_{j+1}-z_j)\\
		\sum_{j=0}^{n-1}\frac{1}{\beta_{j+1}}(z_{j+1}-z_j)+f_{\beta}(z;-kz_n+u^*)\end{bmatrix}
\end{align*}
and
\begin{align*}
	\sigma(z)=\begin{bmatrix}\mathbf{0}_{d\times m}\\\vdots\\ \mathbf{0}_{d\times m}\\
		g_{\beta}(z)\end{bmatrix}.
\end{align*}
	Then system \eqref{n-order system for z} with control \eqref{controller in z} takes the following more compact form:
	\begin{equation}
		\label{SDE}
		\mathrm{d}z(t)=b(z(t))\mathrm{d}t+\sigma(z(t))\mathrm{d}B_t.
	\end{equation}
	From the definition  of $b(\cdot)$, we can obtain
\begin{equation}
		\label{1 in Theorem uniform bound}
	\begin{split}
			z^{\top}b(z)
			&=\sum_{i=0}^{n-1}\sum_{j=0}^i\frac{1}{\beta_{j+1}}z_i^{\top}(z_{j+1}-z_j)+\sum_{j=0}^{n-1}\frac{1}{\beta_{j+1}}z_n^{\top}(z_{j+1}-z_j)+z_n^{\top}f_{\beta}(z;-kz_n+u^*)\\
			&=-\frac{1}{\beta_1}|z_0|^2-\sum_{i=1}^{n-1}\Big(\frac{1}{\beta_{i+1}}-\frac{1}{\beta_i}\Big)|z_i|^2+
			\beta_n^{-1}|z_n|^2-\frac{1}{\beta_1}\sum_{i=2}^nz_0^{\top}z_i+\sum_{i=1}^{n-1}\frac{1}{\beta_i} z_i^{\top}z_{i+1}\\
			&\ \ \ \ -\sum_{i=1}^{n-2}\Big(\frac{1}{\beta_{i+1}}-\frac{1}{\beta_i}\Big)\sum_{j=i+2}^nz_i^{\top}z_j+z_n^{\top}f_{\beta}(z;-kz_n+u^*).
		\end{split}
	\end{equation}
	Note that $\beta_1>\beta_2>\cdots>\beta_n>0$. Then we have
	\begin{equation*}
		\Big|-\frac{1}{\beta_1}z_0^{\top}z_i\Big|\leq \frac{1}{n\beta_1}|z_0|^2+\frac{n}{4\beta_1}|z_i|^2, \ \ \ \ \Big|\frac{1}{\beta_i}z_i^{\top}z_{i+1}\Big|\leq \frac{1}{n\beta_i}|z_i|^2+\frac{n}{4\beta_i}|z_{i+1}|^2
	\end{equation*}
	and
	\begin{equation*}
		\begin{split}
			\Big|-\Big(\frac{1}{\beta_{i+1}}-\frac{1}{\beta_i}\Big)z_i^{\top}z_j\Big|
			\leq \frac{1}{n}\Big(\frac{1}{\beta_{i+1}}-\frac{1}{\beta_i}\Big)|z_i|^2+\frac{n}{4}\Big(\frac{1}{\beta_{i+1}}-\frac{1}{\beta_i}\Big)|z_j|^2,
		\end{split}
    \end{equation*}
which yields that
\begin{align*}
\Big|-\frac{1}{\beta_1}\sum_{i=2}^nz_0^{\top}z_i\Big|\leq  \frac{n-1}{n\beta_1}|z_0|^2+\frac{n}{4\beta_1}\sum_{i=2}^n |z_i|^2,
\end{align*}
\begin{align*}
\bigg|\sum_{i=1}^{n-1}\frac{1}{\beta_i}z_i^{\top}z_{i+1}\bigg|\leq  \sum_{i=1}^{n-1}\frac{1}{n\beta_i}|z_i|^2+
\sum_{i=2}^n \frac{n}{4\beta_{i-1}}|z_i|^2,
\end{align*}
and
\begin{align*}
&\ \ \ \ \bigg|\sum_{i=1}^{n-2}\Big(\frac{1}{\beta_{i+1}}\!-\!\frac{1}{\beta_i}\Big)\!\sum_{j=i+2}^n z_i^{\top}z_j\bigg|\\
&\le \sum_{i=1}^{n-2}\frac{n-(i+1)}{n}\Big(\frac{1}{\beta_{i+1}}-\frac{1}{\beta_i}\Big)|z_i|^2+\sum_{j=3}^n\sum_{i=1}^{j-2}\frac{n}{4}\Big(\frac{1}{\beta_{i+1}}-\frac{1}{\beta_i}\Big)|z_j|^2\\
&=\sum_{i=1}^{n-2}\frac{n-(i+1)}{n}\Big(\frac{1}{\beta_{i+1}}-\frac{1}{\beta_i}\Big)|z_i|^2+\sum_{j=3}^n\frac{n}{4}\Big(\frac{1}{\beta_{j-1}}-\frac{1}{\beta_1}\Big)|z_j|^2.
\end{align*}
Consequently, we have
	\begin{equation}
		\label{2 in Theorem uniform bound}
		\begin{split}
		 &\ \ \ \ -\frac{1}{\beta_1}\sum_{i=2}^nz_0^{\top}z_i+\sum_{i=1}^{n-1}\frac{1}{\beta_i}z_i^{\top}z_{i+1}-\sum_{i=1}^{n-2}\Big(\frac{1}{\beta_{i+1}}-\frac{1}{\beta_i}\Big)\sum_{j=i+2}^n z_i^{\top}z_j\\
		 &\leq \frac{n-1}{n\beta_1}|z_0|^2+\sum_{i=1}^{n-1}\frac{1}{n\beta_i}|z_i|^2+\sum_{i=2}^n\frac{n}{2\beta_{i-1}}|z_i|^2+\sum_{i=1}^{n-2}\frac{n-(i+1)}{n}\Big(\frac{1}{\beta_{i+1}}-\frac{1}{\beta_i}\Big)|z_i|^2.
		\end{split}
	\end{equation}
Then it follows from \eqref{1 in Theorem uniform bound} and \eqref{2 in Theorem uniform bound} that
	\begin{equation*}
		\begin{split}
			z^{\top}b(z)&\leq -\frac{1}{n\beta_1}|z_0|^2-\Big(\frac{2}{n\beta_2}-\frac{3}{n\beta_1}\Big)|z_1|^2-\sum_{i=2}^{n-1}\Big(\frac{i+1}{n\beta_{i+1}}-\frac{i+2}{n\beta_i}-\frac{n}{2\beta_{i-1}}\Big)|z_i|^2\\
			&\ \ \ \ +\Bigl(\frac{1}{\beta_n}+\frac{n}{2\beta_{n-1}}\Bigr)|z_n|^2+z_n^{\top}f_{\beta}(z;-kz_n+u^*).
		\end{split}
	\end{equation*}
	Since $\beta_1,\cdots,\beta_n$ are given by formula \eqref{1616}, it is easy to check that
	\begin{equation*}
		\frac{1}{n\beta_1}>\lambda+8M^2, \ \ \ \ \frac{2}{n\beta_2}-\frac{3}{n\beta_1}\geq \frac{2}{\beta_1}-\frac{3}{n\beta_1}>\frac{1}{n\beta_1}>\lambda+8M^2,
	\end{equation*}
	and for $2\leq i\leq n-1$,
	\begin{equation*}
			\frac{i+1}{n\beta_{i+1}}-\frac{i+2}{n\beta_i}-\frac{n}{2\beta_{i-1}}>\frac{1}{\beta_i}-\frac{n}{2\beta_{i-1}}>\frac{n}{2\beta_{i-1}}>\frac{1}{n\beta_1}>\lambda+8M^2.
	\end{equation*}
	Note also that
	\begin{equation*}
		\frac{1}{\beta_n}+\frac{n}{2\beta_{n-1}}+\lambda+8M^2<\frac{1}{\beta_n}+\frac{1}{2\beta_n}+\frac{1}{n\beta_1}<\frac{2}{\beta_n},
	\end{equation*}
	then we have
	\begin{equation*}
			z^{\top}b(z)\leq -(\lambda+8M^2) |z|^2+\frac{2}{\beta_n}|z_n|^2
			+z_n^{\top}f_{\beta}(z,-kz_n+u^*).
	\end{equation*}
	Note that
	\begin{equation}\label{0417-1}
		\begin{split}
			z_n^{\top}f_{\beta}(z;-kz_n+u^*)
			&=\widehat{\beta}_nz_n^{\top}f(z^{\beta}+z^*;-kz_n+u^*)\\
			&= \widehat{\beta}_n\bigl(z_n^{\top}f(z^{\beta}+z^*;u^*)-kz_n^{\top}\theta(z^{\beta};-kz_n)z_n\bigr),
		\end{split}
	\end{equation}
	where $\theta(z^{\beta};-kz_n)$ is defined as in \eqref{derivative matrix theta}. Recall Assumptions \ref{A 1}-\ref{A 2}, we derive
	\begin{equation*}
		\begin{split}
			z_n^{\top}f_{\beta}(z;-kz_n+u^*)\leq L\widehat{\beta}_n|z_n||z^{\beta}|-k\widehat{\beta}_n|z_n|^2.
		\end{split}
	\end{equation*}
	Since $0<\beta_2,\cdots, \beta_n<1$, it is easy to check that
    $$|z^{\beta}|^2\leq\frac{1}{\widehat{\beta}_n^2} \sum_{i=1}^n|z_i-z_{i-1}|^2  \le \frac{4}{\widehat{\beta}_n^2}|z|^2.$$
	Hence
	\begin{equation*}
			z_n^{\top}f_{\beta}(z;-kz_n+u^*)
			\leq 2L|z_n||z|-k\widehat{\beta}_n|z_n|^2
			\leq \frac{\lambda+8M^2}{2}|z|^2-\Bigl(k\widehat{\beta}_n-\frac{2L^2}{\lambda+8M^2}\Bigr)|z_n|^2.
	\end{equation*}
	Since
	\begin{equation*}
		\begin{split}
			k\widehat{\beta}_n>\frac{1}{\widehat{\beta}_n}\bigg(1+3L+\frac{2L^2}{\lambda+8M^2}\bigg)
			>\frac{1}{\widehat{\beta}_n}+\frac{2L^2}{\lambda+8M^2}>\frac{2}{\beta_n}+\frac{2L^2}{\lambda+8M^2},
		\end{split}
	\end{equation*}
	we obtain that
	\begin{equation}
		\label{1207-1}
		z^{\top}b(z)\leq -\frac{\lambda+8M^2}{2}|z|^2.
	\end{equation}
	In addition, by the definition of $g_{\beta}$, we know that
	\begin{equation}
		\label{1207-2}
		\|g_{\beta}(z)\|_{{\rm HS}}^2=\widehat{\beta}_n^2\|g(z^{\beta}+z^*)\|_{{\rm HS}}^2\leq 8M^2|z|^2+2\widehat{\beta}_n^2\|g(z^*)\|_{{\rm HS}}^2.
	\end{equation}
	
	Now consider the Lyapunov function $V(z)=|z|^2$, \eqref{1207-1} and \eqref{1207-2} yield
	\begin{equation}\label{exist}
		\begin{split}
			\mathcal{L}V(z)=2z^{\top}b(z)+\|g_{\beta}(z)\|_{{\rm HS}}^2
			\leq -\lambda V(z)+2\widehat{\beta}_n^2\|g(z^*)\|_{{\rm HS}}^2.
		\end{split}
	\end{equation}
	By Lemma \ref{lemma upper bound}, we know that there exists a unique continuous solution $z(t)$ of SDE \eqref{SDE} and
	\begin{equation}
		\label{1208-1}
		\mathbf{E}[|z(t)|^2]\leq |z(0)|^2e^{-\lambda t}+\frac{2\widehat{\beta}_n^2}{\lambda}\|g(z^*)\|_{{\rm HS}}^2.
	\end{equation}
	Recall transformations \eqref{5.1}, \eqref{transform from y to z} and the fact that $k_0=k\geq 1$, we have
	\begin{equation*}
		|z_0(0)|^2=|y_0(0)|^2=\Big|\frac{u^*}{k_0}\Big|^2\leq |u^*|^2,
	\end{equation*}
	and for $1\leq i\leq n$,
	\begin{equation*}
		\begin{split}
			|z_i(0)|^2=\Big|y_0(0)+\sum_{j=1}^i\widehat{\beta}_jy_j(0)\Big|^2
			\leq \Bigl(\Big|\frac{u^*}{k_0}\Big|+\sum_{j=1}^i|y_i(0)|\Bigr)^2
			\leq 2|u^*|^2+2i|x(0)-z^*|^2.
		\end{split}
	\end{equation*}
	Hence
	\begin{equation}\label{bound of z}
		|z(0)|^2\leq 2n^2\big( |x(0)-z^*|^2+|u^*|^2 \big).
	\end{equation}
	Then for $1\leq i\leq n$, we derive that
	\begin{equation}\label{0429}
		\begin{split}
			\mathbf{E}\bigl[|y_i(t)|^2\bigr]=&\frac{1}{\widehat{\beta}_i^2}\mathbf{E}\bigl[|z_i(t)-z_{i-1}(t)|^2\bigr]\leq \frac{2}{\widehat{\beta}_n^2}\mathbf{E}\bigl[|z(t)|^2\bigr]\\
			\leq& \frac{4n^2}{\widehat{\beta}_n^2}\big( |x(0)-z^*|^2+|u^*|^2 \big)e^{-\lambda t}+\frac{4}{\lambda}\|g(z^*)\|_{{\rm HS}}^2.
		\end{split}
	\end{equation}
	Thus \eqref{limit bound of x} follows.

	Now let us consider the case that $\big\|\frac{\partial f}{\partial u}\big\|\leq R$, i.e.,
	\begin{equation}\label{0418-1}
		\bigg|\frac{\partial f}{\partial u}(x;u)v\bigg|\leq R|v|, \ \text{ for all } \ x\in \mathbb{R}^{nd}, \ u,v\in \mathbb{R}^d.
	\end{equation}
	For $\tilde{y}:=\big[\tilde{y}_1^{\top},\cdots,\tilde{y}_n^{\top}\big]^{\top}\in \mathbb{R}^{nd}$, set
	\begin{align*}
		\tilde{b}(t,\tilde{y})=\begin{bmatrix}\tilde{y}_2\\
			\vdots\\
			\tilde{y}_n\\
			f\big(\tilde{y}+z^*;u(t)\big)\end{bmatrix},\
			\tilde{\sigma}(\tilde{y})=\begin{bmatrix}\mathbf{0}_{d\times m}\\
				\vdots\\
				\mathbf{0}_{d\times m}\\
				g(\tilde{y}+z^*)\end{bmatrix},
	\end{align*}
	where
	\begin{equation}\label{0503-1}
		u(t)=-\sum_{i=0}^nk_iy_i(t)+u^*,
	\end{equation}
	and
	\begin{equation}\label{0503-2}
		Y(t):=\bigl[y_0^{\top}(t),y_1^{\top}(t),\cdots,y_n^{\top}(t)\bigr]^{\top}:=\bigl[y_0^{\top}(t);y^{\top}(t)\bigr]^{\top}
	\end{equation}
	is the solution of the system \eqref{n-order system for y}, which is regarded as a fixed adapted stochastic process in the definition of $\tilde{b}$. Let us consider the following SDE:
	\begin{equation}\label{SDE-1}
		\begin{cases}
			d\tilde{y}(t)=\tilde{b}(t,\tilde{y}(t))\mathrm{d}t+\tilde{\sigma}(\tilde{y}(t))\mathrm{d}B_t,\\
			\tilde{y}(0)=y(0),
		\end{cases}
	\end{equation}
	where $y(0):=\big[y_1^{\top}(0),\cdots,y_n^{\top}(0)\big]^{\top}$ is the same as the initial condition of system \eqref{n-order system for y}. Then by the uniqueness solution of system \eqref{n-order system for y} and SDE \eqref{SDE-1}, we know that their solutions are the same, i.e. $\tilde{y}(t)=y(t)$ for all $t\geq 0$.

	By the construction of $\tilde{b}$, we have
	\begin{equation}\label{0503-3}
		\begin{split}
			2\tilde{y}^{\top}\tilde{b}(t,\tilde{y})&=\sum_{i=1}^{n-1}2\tilde{y}_i^{\top}\tilde{y}_{i+1}+2\tilde{y}_n^{\top}f\big(\tilde{y}+z^*;u(t)\big)\\
			&\geq -\sum_{i=1}^{n-1}\bigl(|\tilde{y}_i|^2+|\tilde{y}_{i+1}|^2\bigr)+2\tilde{y}_n^{\top}f(\tilde{y}+z^*,u^*)
			+2\tilde{y}_n^{\top}\theta(\tilde{y},u(t))(u(t)-u^*),
		\end{split}
	\end{equation}
	where $\theta(\tilde{y},u(t))$ is defined as in \eqref{derivative matrix theta}. Note that Assumptions \ref{A 1}-\ref{A 2} and \eqref{0418-1} yield
	\begin{equation*}
		2\tilde{y}_n^{\top}f(\tilde{y}+z^*,u^*)\geq -2L|\tilde{y}_n||\tilde{y}|\geq -2L|\tilde{y}|^2,
	\end{equation*}
	and
	\begin{equation*}
		\begin{split}
			&\ \ \ \ 2\tilde{y}_n^{\top}\theta(\tilde{y},u(t))(u(t)-u^*)\\
			&\geq -2R|\tilde{y}_n|\bigg|\sum_{i=0}^nk_iy_i(t)\bigg|
			\geq -\frac{16(n+1)R^2|\mathbf{k}|^2}{\lambda}|\tilde{y}_n|^2-\frac{\lambda}{16(n+1)}|Y(t)|^2,
		\end{split}
	\end{equation*}
	where $\mathbf{k}=[k_0,k_1,\cdots,k_n]^{\top}$.
	Hence
	\begin{equation}\label{0418-2}
		\begin{split}
			2\tilde{y}^{\top}\tilde{b}(t,\tilde{y})\geq -\left(2+2L+\frac{16(n+1)R^2|\mathbf{k}|^2}{\lambda}\right)|\tilde{y}|^2-\frac{\lambda}{16(n+1)}|Y(t)|^2.
		\end{split}
	\end{equation}
	On the other hand, Assumption \ref{A 1} also implies
	\begin{equation}\label{0418-3}
		\begin{split}
			\|\tilde{\sigma}(\tilde{y})\|_{{\rm HS}}^2=&\|g(\tilde{y}+z^*)\|_{{\rm HS}}^2=\|g(z^*)+g(\tilde{y}+z^*)-g(z^*)\|_{{\rm HS}}^2\\
			\geq& \bigl(\|g(z^*)\|_{{\rm HS}}-\|g(\tilde{y}+z^*)-g(z^*)\|_{{\rm HS}}\bigr)^2\\
			\geq& \frac{1}{2}\|g(z^*)\|_{{\rm HS}}^2-\|g(\tilde{y}+z^*)-g(z^*)\|_{{\rm HS}}^2\\
			\geq& \frac{1}{2}\|g(z^*)\|_{{\rm HS}}^2-M^2|\tilde{y}|^2,
		\end{split}
	\end{equation}
	where we have used the following inequality in \eqref{0418-3}:
	\begin{equation*}
		(a-b)^2\geq \frac{1}{2}a^2-b^2.
	\end{equation*}
	Then for the Lyapunov function $V(\tilde{y})=|\tilde{y}|^2$, \eqref{0418-2} and \eqref{0418-3} indicate
	\begin{equation*}
		\begin{split}
			\tilde{\mathcal{L}}_tV(\tilde{y})=2\tilde{y}^{\top}\tilde{b}(t,\tilde{y})+\|\tilde{\sigma}(\tilde{y})\|_{{\rm HS}}^2
			\geq -w|\tilde{y}|^2+\frac{1}{2}\|g(z^*)\|_{{\rm HS}}^2-\frac{\lambda}{16(n+1)}|Y(t)|^2,
		\end{split}
	\end{equation*}
	where
	$$w=2+2L+M^2+\frac{16(n+1)R^2|\mathbf{k}|^2}{\lambda}.$$
According to \eqref{1208-1}, \eqref{bound of z} and transformation \eqref{transform from y to z}, it can be obtained that $y_0(t)$ has the same estimate as in \eqref{0429}. Thus we have
	\begin{equation*}
		\begin{split}
			\frac{1}{2}\|g(z^*)\|_{{\rm HS}}^2-\frac{\lambda}{16(n+1)}\mathbf{E}\bigl[|Y(t)|^2\bigr]
			\geq \frac{1}{4}\|g(z^*)\|_{{\rm HS}}^2-\frac{\lambda n^2k_0^2}{4k_n^2}\bigl(|x(0)-z^*|^2+|u^*|^2\bigr)e^{-\lambda t}.
		\end{split}
	\end{equation*}
	Then it follows from Lemma \ref{lemma lower bound} and the fact $y(t)=\tilde{y}(t)$ that
	\begin{equation}\label{0430-1}
		\begin{split}
			\mathbf{E}\bigl[|y(t)|^2\bigr]
			\geq& |y(0)|^2e^{-wt}+\frac{1}{4w}\bigl(1-e^{-wt}\bigr)\|g(z^*)\|_{{\rm HS}}^2\\
			&-\frac{\lambda n^2k_0^2}{4k_n^2}\bigl(|x(0)-z^*|^2+|u^*|^2\bigr) e^{-w t}\int_{0}^{t}e^{(w-\lambda)s}\mathrm{d}s.
		\end{split}
	\end{equation}
	Note that
	\begin{equation}\label{0430-2}
		e^{-w t}\int_{0}^{t}e^{(w-\lambda)s}\mathrm{d}s\leq 1_{\{w\leq \lambda\}}e^{-w t}t+1_{\{w>\lambda\}}\frac{e^{-\lambda t}}{w-\lambda},
	\end{equation}
	then
	\begin{equation*}
		e^{-w t}\int_{0}^{t}e^{(w-\lambda)s}\mathrm{d}s \to 0 \text{ as  } t\to \infty.
	\end{equation*}
	Hence \eqref{limit lower bound of x} follows with $C_3=\frac{1}{4w}$.
\end{proof}

Next we give the proof of Proposition \ref{prop of lower bound}.

\begin{proof}[\textnormal{Proof of Proposition \ref{prop of lower bound}}.]
	For any $p> 1$, we first give the $2p$-th moment estimate of the solution $(y_0(t),y_1(t),\cdots,y_n(t))$ to system \eqref{n-order system for y}.
	Recall transformations \eqref{5.1}, \eqref{transform from y to z} and their corresponding systems \eqref{n-order system for y}, \eqref{n-order system for z}. Let us consider the Lyapunov function $V_p(z):=|z|^{2p}=V^p(x)$, then \eqref{exist} and the definition \eqref{definition of operator L} of $\mathcal{L}$ imply
	\begin{equation*}
		\begin{split}
			\mathcal{L}V_p(z)&=pV^{p-1}(z)\mathcal{L}V(z)
			+\frac{1}{2}p(p-1)V^{p-2}(z)\bigg|\sigma^{\top}(z)\frac{\partial V(z)}{\partial z}\bigg|^2\\
			&\leq -\lambda p V_p(z)+2pV^{p-1}(z)\widehat{\beta}_n^2\|g(z^*)\|_{{\rm HS}}^2
			+2p(p-1)V^{p-2}(z)|g^{\top}_{\beta}(z)z_n|^2.
		\end{split}
	\end{equation*}
	By \eqref{1207-2} and the fact that $0\leq g_{\beta}(z)g^{\top}_{\beta}(z)\leq \|g_{\beta}(z)\|_{{\rm HS}}^2 I_d$,
	we have
	\begin{equation*}
		\begin{split}
			|g_{\beta}^{\top}(z)z_n|^2=z_n^{\top}g_{\beta}(z)g^{\top}_{\beta}(z)z_n
			\leq \|g_{\beta}(z)\|_{{\rm HS}}^2|z|^2
			\leq 8M^2|z|^4+2\widehat{\beta}_n^2\|g(z^*)\|_{{\rm HS}}^2|z|^2.
		\end{split}
	\end{equation*}
	Thus
	\begin{equation*}
		\begin{split}
			\mathcal{L}V_p(z)&\leq -p\bigl(\lambda-16(p-1)M^2\bigr)V_p(z)
			+4p^2\widehat{\beta}_n^2\|g(z^*)\|_{{\rm HS}}^2 V^{p-1}(z)\\
			&\leq -\frac{p}{2}\bigl(\lambda-16(p-1)M^2\bigr)V_p(z)
			+4^pp^{2p-1}a^{-p}\widehat{\beta}_n^{2p}\|g(z^*)\|_{{\rm HS}}^{2p},
		\end{split}
	\end{equation*}
	where $a=\big[\frac{p^2(\lambda-16(p-1)M^2)}{2(p-1)}\big]^{\frac{p-1}{p}}$ and we have used the following inequality:
	\begin{equation*}
		\begin{split}
			4p^2\widehat{\beta}_n^2\|g(z^*)\|_{{\rm HS}}^2 V^{p-1}(z)\leq \frac{p-1}{p}a^{\frac{p}{p-1}}V^p(z)
			+\frac{1}{p}a^{-p}\bigl(4p^2\widehat{\beta}_n^2\|g(z^*)\|_{{\rm HS}}^2\bigr)^p.
		\end{split}
	\end{equation*}
	Then Lemma \ref{lemma upper bound} yields
	\begin{equation*}
		\begin{split}
			\mathbf{E}\bigl[|z(t)|^{2p}\bigr]\leq |z(0)|^{2p}e^{-\frac{p}{2}(\lambda-16(p-1)M^2)t}
			+\frac{2^{2p+1}p^{2p-2}\widehat{\beta}_n^{2p}}{a^p(\lambda-16(p-1)M^2)}\|g(z^*)\|_{{\rm HS}}^{2p}.
		\end{split}
	\end{equation*}
	Therefore, according to the transformation \eqref{transform from y to z} and the fact that
	\begin{equation*}
		1 \leq \frac{k_0}{k_i}=\frac{1}{\widehat{\beta}_i}\leq \frac{1}{\widehat{\beta}_n}=\frac{k_0}{k_n}, \ \ 1\leq i\leq n,
	\end{equation*}
	we conclude that for all $i\in \{0,1,\cdots, n\}$,
	\begin{equation}\label{0419-1}
		\begin{split}
			\mathbf{E}\bigl[|y_i(t)|^{2p}\bigr]\leq \frac{2^pk_0^{2p}}{k_n^{2p}}|z(0)|^{2p}e^{-\frac{p}{2}(\lambda-16(p-1)M^2)t}
			+\frac{4^{p+1}p^{2p-2}}{a^p(\lambda-16(p-1)M^2)}\|g(z^*)\|_{{\rm HS}}^{2p}.
		\end{split}
	\end{equation}

	Now for $\widehat{y}:=\big[\widehat{y}_1^{\top},\cdots,\widehat{y}_n^{\top}\big]^{\top}\in \mathbb{R}^{nd}$, we set
	\begin{align*}
		\widehat{b}(t,\widehat{y})=\begin{bmatrix}
			\widehat{y}_2\\
			\vdots\\
			\widehat{y}_n\\
			f\big(y(t)+z^*;u(t)\big)\end{bmatrix},\
			\widehat{\sigma}(\widehat{y})=\begin{bmatrix}\mathbf{0}_{d\times m}\\
				\vdots\\
				\mathbf{0}_{d\times m}\\
				g(\widehat{y}+z^*)\end{bmatrix},
	\end{align*}
where $u(t)$ and $Y(t)=[y_0^{\top}(t),y^{\top}(t)]^{\top}$ are defined as in \eqref{0503-1}-\eqref{0503-2}.
Similar to the proof of Theorem \ref{Theorem uniform bounded with arbitrary small limit}, we consider the following SDE:
	\begin{equation}\label{SDE-2}
		\begin{cases}
			d\widehat{y}(t)=\widehat{b}(t,\widehat{y}(t))\mathrm{d}t+\widehat{\sigma}(\widehat{y}(t))\mathrm{d}B_t,\\
			\widehat{y}(0)=y(0).
		\end{cases}
	\end{equation}
	Then SDE \eqref{SDE-2} has a unique solution $\widehat{y}(t)$ which is also the same as $y(t)$.

	Note that $\lambda>16\gamma M^2$. Then \eqref{0419-1} and the fact that
	\begin{equation}\label{0419-3}
		\biggl(\sum_{i=1}^na_i\biggr)^p\leq C(n,p)\sum_{i=1}^na_i^p, \ \ n\geq 2, \ p>0, \ a_i\geq 0,
	\end{equation}
	indicate
	\begin{equation*}
		\begin{split}
			\mathbf{E}\bigl[|Y(t)|^{2(\gamma+1)}\bigr]\leq C_5|z(0)|^{2(\gamma+1)}e^{-\delta t}+C_6\|g(z^*)\|_{{\rm HS}}^{2(\gamma+1)},
		\end{split}
	\end{equation*}
	where
	\begin{equation*}
		\begin{split}
			&C_5=\frac{(n+1)C(n+1,\gamma+1)2^{\gamma+1}k_0^{2(\gamma+1)}}{k_n^{2(\gamma+1)}}, \\
			&C_6=\frac{(n+1)C(n+1,\gamma+1)4^{\gamma+2}(\gamma+1)^{2\gamma}}{a^{\gamma+1}(\lambda-16\gamma M^2)}, \\
			&\ \delta  \ =\frac{\gamma+1}{2}(\lambda-16\gamma M^2)>0.
		\end{split}
	\end{equation*}
	On the other hand, according to \eqref{1208-1} and transformation \eqref{transform from y to z}, $Y(t)$ has the following second moment estimate:
	\begin{equation*}
		\mathbf{E}\bigl[|Y(t)|^2\bigr]\leq C_1'|z(0)|^2e^{-\lambda t}+C_2'\|g(z^*)\|_{{\rm HS}}^2,
	\end{equation*}
	where
	\begin{equation*}
		C_1'=\frac{2(n+1)k_0^2}{k_n^2}, \ C_2'=\frac{4(n+1)}{\lambda}.
	\end{equation*}
	Similar to \eqref{0503-3}, we have
	\begin{equation}\label{0419-2}
		\begin{split}
			2\widehat{y}^{\top}\widehat{b}(t,\widehat{y})
			\geq -2|\widehat{y}|^2+2\widehat{y}_n^{\top}f(y(t)+z^*,u^*)
			-2\widehat{y}_n^{\top}\theta(y(t),u(t))\sum_{i=0}^nk_iy_i(t),
		\end{split}
	\end{equation}
	where
	\begin{equation}\label{0503-4}
		\theta(y(t),u(t))=\int_0^1 \frac{\partial f}{\partial \bar u}(y(t)+z^*;\bar u)dr,~\bar u=u^*+r(u(t)-u^*).
	\end{equation}
	Note that the Lipschitz condition of $f$ and the fact that $f(z^*,u^*)=0$ yield
	\begin{equation}\label{0419-4}
		\begin{split}
		2\widehat{y}_n^{\top}f(y(t)+z^*,u^*)
		\geq -2L|\widehat{y}_n||y(t)|
        \geq -8C_2'L^2|\widehat{y}_n|^2-\frac{1}{8C_2'}|Y(t)|^2.
		\end{split}
	\end{equation}
	Now let us deal with the third term in the right hand side of \eqref{0419-2}. Note first that
	\begin{equation*}
		|u(t)-u^*|=\bigg|\sum_{i=0}^nk_iy_i(t)\bigg|\leq |\mathbf{k}||Y(t)|,
	\end{equation*}
	where $\mathbf{k}:=[k_0,k_1,\cdots,k_n]^{\top}$. Then the polynomial growth condition \eqref{polynomial growth condition}  of $\frac{\partial f}{\partial u}$ and \eqref{0503-4} imply
	\begin{equation}\label{0419-5}
		\begin{split}
			2\widehat{y}_n^{\top}\theta(y(t),u(t))\sum_{i=0}^nk_iy_i(t)
			&\leq2|\mathbf{k}|R\Bigl(1+|y(t)+z^*|^{\gamma}+\bigl(|u^*|+|\mathbf{k}||Y(t)|\bigr)^{\gamma}\Bigr)|\widehat{y}_n||Y(t)|\\
			&\leq 2^{\gamma+1}|\mathbf{k}|R\bigl(1+|y^*|^{\gamma}+|u^*|^{\gamma}\bigr)|\widehat{y}_n||Y(t)|\\
			&\ \ \ \ +2^{\gamma+1}|\mathbf{k}|R\bigl(1+|\mathbf{k}|^{\gamma}\bigr)|\widehat{y}_n||Y(t)|^{\gamma+1}\\
			&\leq 2^{2\gamma+3}C_2'|\mathbf{k}|^2R^2\bigl(1+|y^*|^{\gamma}+|u^*|^{\gamma}\bigr)^2|\widehat{y}_n|^2\\
			& \ \ \ \ +2^{2\gamma+3}C_6|\mathbf{k}|^2R^2\big(1+|\mathbf{k}|^{\gamma}\big)^2\|g(z^*)\|_{{\rm HS}}^{2\gamma}|\widehat{y}_n|^2\\
			&\ \ \ \ +\frac{1}{8C_2'}|Y(t)|^2+\frac{1}{8C_6\|g(z^*)\|_{{\rm HS}}^{2\gamma}}|Y(t)|^{2(\gamma+1)}.
		\end{split}
	\end{equation}
	On the other hand, similar to \eqref{0418-3}, we also have
	\begin{equation}\label{0419-6}
		\|\widehat{\sigma}(\widehat{y})\|_{{\rm HS}}^2\geq \frac{1}{2}\|g(z^*)\|_{{\rm HS}}^2-M^2|\widehat{y}|^2.
	\end{equation}
	Then it follows from \eqref{0419-2}, \eqref{0419-4}, \eqref{0419-5} and \eqref{0419-6} that
	\begin{equation*}
		\begin{split}
			\widehat{\mathcal{L}}_tV(\widehat{y})&=2\widehat{y}^{\top}\widehat{b}(t,\widehat{y})+\|\widehat{\sigma}(\widehat{y})\|_{{\rm HS}}^2\\
			&\geq -A\bigl(1+|y^*|^{2\gamma}+|u^*|^{2\gamma}+\|g(z^*)\|_{{\rm HS}}^{2\gamma}\bigr)|\widehat{y}|^2\\
			&\ \ \ \ +\frac{1}{2}\|g(z^*)\|_{{\rm HS}}^2-\frac{1}{4C_2'}|Y(t)|^2
			-\frac{1}{8C_6\|g(z^*)\|_{{\rm HS}}^{2\gamma}}|Y(t)|^{2(\gamma+1)},
		\end{split}
	\end{equation*}
	where the positive constant $A$ only depend on $(k_0,\cdots,k_n,\lambda,L,M,R,\gamma,n)$. Note that $\widehat{y}(t)=y(t)$ for all $t\geq 0$ and
	\begin{equation*}
		\begin{split}
			&\ \ \ \ \frac{1}{2}\|g(z^*)\|_{{\rm HS}}^2-\frac{1}{4C_2'}\mathbf{E}\bigl[|Y(t)|^2\bigr]
			-\frac{1}{8C_6\|g(z^*)\|_{{\rm HS}}^{2\gamma}}\mathbf{E}\bigl[|Y(t)|^{2(\gamma+1)}\bigr]\\
			&\geq \frac{1}{8}\|g(z^*)\|_{{\rm HS}}^2-\frac{C_1'}{4C_2'}|z(0)|^2e^{-\lambda t}
			-\frac{C_5}{8C_6\|g(z^*)\|_{{\rm HS}}^{2\gamma}}|z(0)|^{2(\gamma+1)}e^{-\delta t},
		\end{split}
	\end{equation*}
	then arguing as in \eqref{0430-1} and \eqref{0430-2}, we conclude that
	\begin{equation*}
		\liminf_{t\to \infty} \mathbf{E}\bigl[|y(t)|^2\bigr]\geq \frac{1}{8A} \frac{\|g(z^*)\|_{{\rm HS}}^2}{1+|y^*|^{2\gamma}+|u^*|^{2\gamma}+\|g(z^*)\|_{{\rm HS}}^{2\gamma}}.
	\end{equation*}
\end{proof}

\section{Simulations}
In this section, we consider a class of third-order nonlinear stochastic systems with parameter uncertainty
	\begin{align}\label{88}\begin{cases}
			\mathrm{d}x_{1}  =x_{2}\mathrm{d}t\\
\mathrm{d}x_{2}  =x_{3}\mathrm{d}t\\
			\mathrm{d}x_{3} =f(x_1,x_2,x_3;u)\mathrm{d}t+g(x_1,x_2,x_3)\mathrm{d}B_t
\end{cases}
	\end{align}
where $(x_1,x_2,x_3)\in\mathbb{R}^3,~u\in\mathbb{R}^1$ and the unknown functions $f$ and $g$ take the following forms
 \begin{align*}
&f=a\sin x_1+bx_2+cx_3+d+u+\mu \tanh (u),~~~g=\sigma.
 \end{align*}
 Here $a$, $b$, $c$, $d$, $\mu$ and $\sigma$ are six \emph{unknown} system parameters which satisfy
\begin{align}\label{40}|a|\vee |b|\vee |c|\le 1/2,~~d\in\mathbb{R},~~\mu\geq 0,~~\sigma \in \mathbb{R}.\end{align}
 From the expressions of $f$ and $g$ and (\ref{40}), it is easy to obtain
 \begin{align*}
 &\Big|\frac{\partial f}{\partial x_1}\Big|=|a\cos x_1|\le 1/2,~~\Big|\frac{\partial f}{\partial x_2}\Big|\le 1/2,~~
 \Big|\frac{\partial f}{\partial x_3}\Big|\le 1/2,\\
&\frac{\partial f}{\partial u}=1+\mu \left(1-\tanh^2 (u)\right)\geq 1,~~\frac{\partial g}{\partial x_1}=0,~~\frac{\partial g}{\partial x_2}=0,
\end{align*}
which in turn gives
$$|f(x;u)-f(y;u)|\le \sqrt{3}|x-y|/2, ~~|g(x)-g(y)|=0.$$
Hence, Assumptions 1-2 hold with $L=\sqrt 3/2$ and $M=0$.
Now, suppose the control variable $u$ takes the form
\begin{align}\label{89}\begin{split}
u(t)&=k_{0} \!\int_{0}^{t}\! e(s)\mathrm{d}s+k_{1} e(t)+k_{2} \dot{e}(t)+k_3\ddot e(t)\\
e(t)&=y^*-x_1(t).\end{split}
\end{align}
Then it follows from Theorem 1 that, there exist positive constants $C_1$, $C_2$ and $\lambda$  such that \begin{align*}
		\mathbf{E}\left[|e(t)|^2\right]\leq C_1\left[|x(0)-z^*|^2+|u^*|^2\right]e^{-\lambda t}+ C_2\|g(z^*)\|_{{\rm HS}}^2,
	\end{align*}
if the extended PID parameters $(k_0,k_1,k_2,k_3)$ satisfy the following inequalities:
\begin{align}\label{a11}
\min\Big\{k_0^2,~k_{1}^2\!-\!2k_{0}k_{2},~k_{2}^2\!-\!2k_{1}k_{3},~ k_3^2\!-\!k_{2}\Big\}>\sqrt 3\bar k/2
\end{align}
where $\bar k=k_0+k_1+k_2+k_3$. We assume the parameters $(k_0,k_1,k_2,k_3)$ are given by
$$k_0=k_3=k,~~~k_1=k_2=2.5k,$$
then (\ref{a11}) reduces to
$$\min\Big\{k^2,1.25k^2,k^2-2.5k\Big\} =k^2-2.5k>\frac{7\sqrt 3}{2} k.$$
Therefore, the control system (\ref{88}) and (\ref{89}) will be globally stable if $$k>(5+7\sqrt 3)/2\approx 8.56.$$
In the following simulations, the extended PID parameters are chosen as $(k_0,k_1,k_2,k_3)=(8.6,21.5,21.5,8.6)$.
In Fig. \ref{fig1}, we illustrate the tracking performance under \emph{different system parameters} $(a,b,c,d,\mu,\sigma)$.  One can see that the given extended PID controller (\ref{89}) has the ability to stabilize and
regulate the control system (\ref{88}), even if the system parameters $(a, b, c, d,\mu,\sigma)$ vary in a wide range. This demonstrates that the extended PID control has strong robustness with respect to the system uncertainties. In  Fig. \ref{fig2}, we want to see how the constant $\sigma $ (which reflects the noise intensity) will affect the tracking performance. It can be seen that the steady-state tracking error will be \emph{small} if $\sigma$  is small. In particular, the tracking error will converge to zero if $\sigma=0$.
In Fig. \ref{fig3}, one can also see that large random disturbance may lead to the high-frequency oscillation of the control input, and increase the variance of the controller input.
 \begin{figure}[htbp]
 \begin{minipage}[t]{1\linewidth}
 \includegraphics[width=3.5in, height=3in]{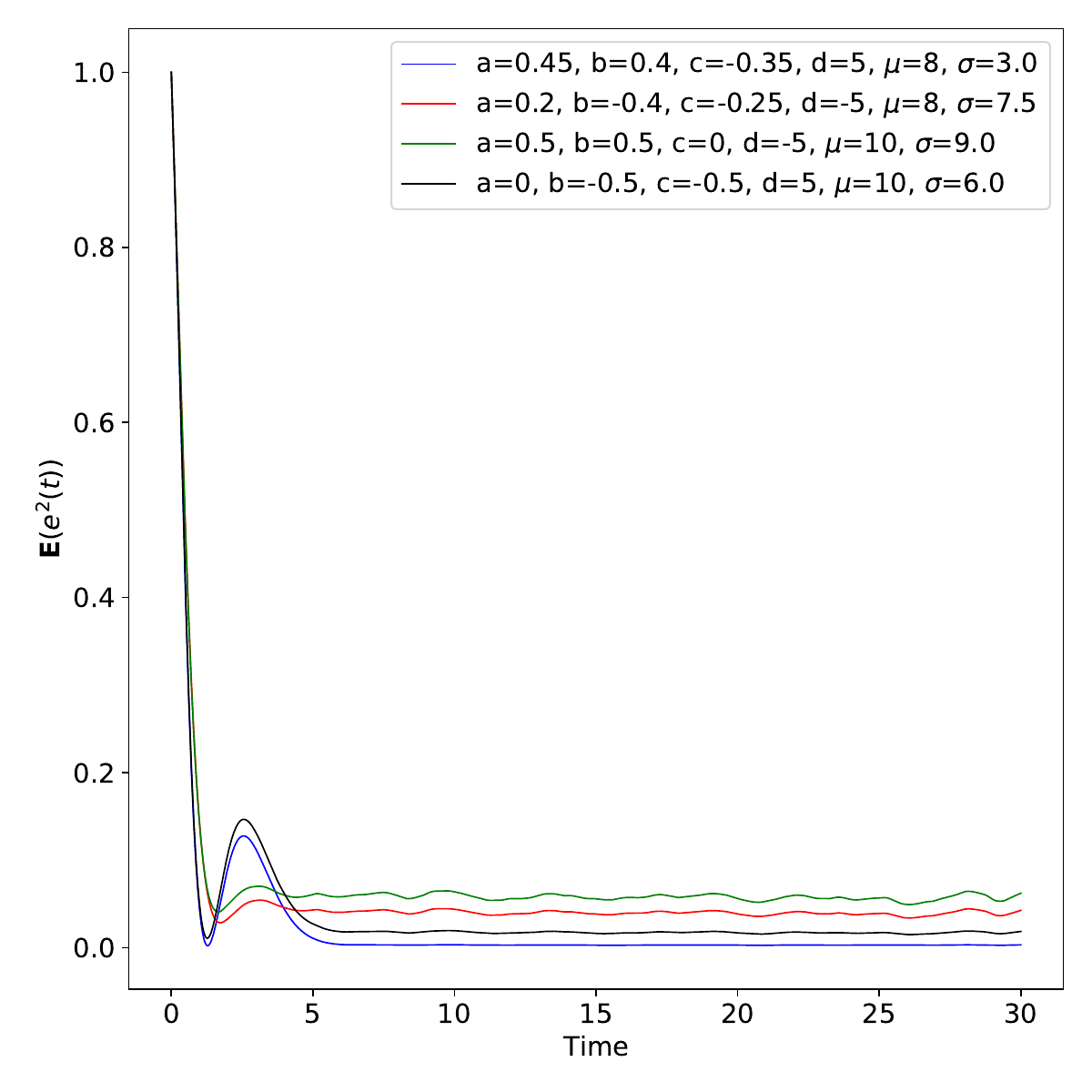}
 \caption{Curves of $\mathbf{E}|e(t)|^2$ under different system parameters. The system initial state $(x_1(0),x_2(0),x_3(0))$ is $(0.5,0.5,0.3)$.}
 \label{fig1}
 \end{minipage}
 \end{figure}

 \begin{figure}[htbp]
 \begin{minipage}[t]{1\linewidth}
 \includegraphics[width=3.5in, height=3in]{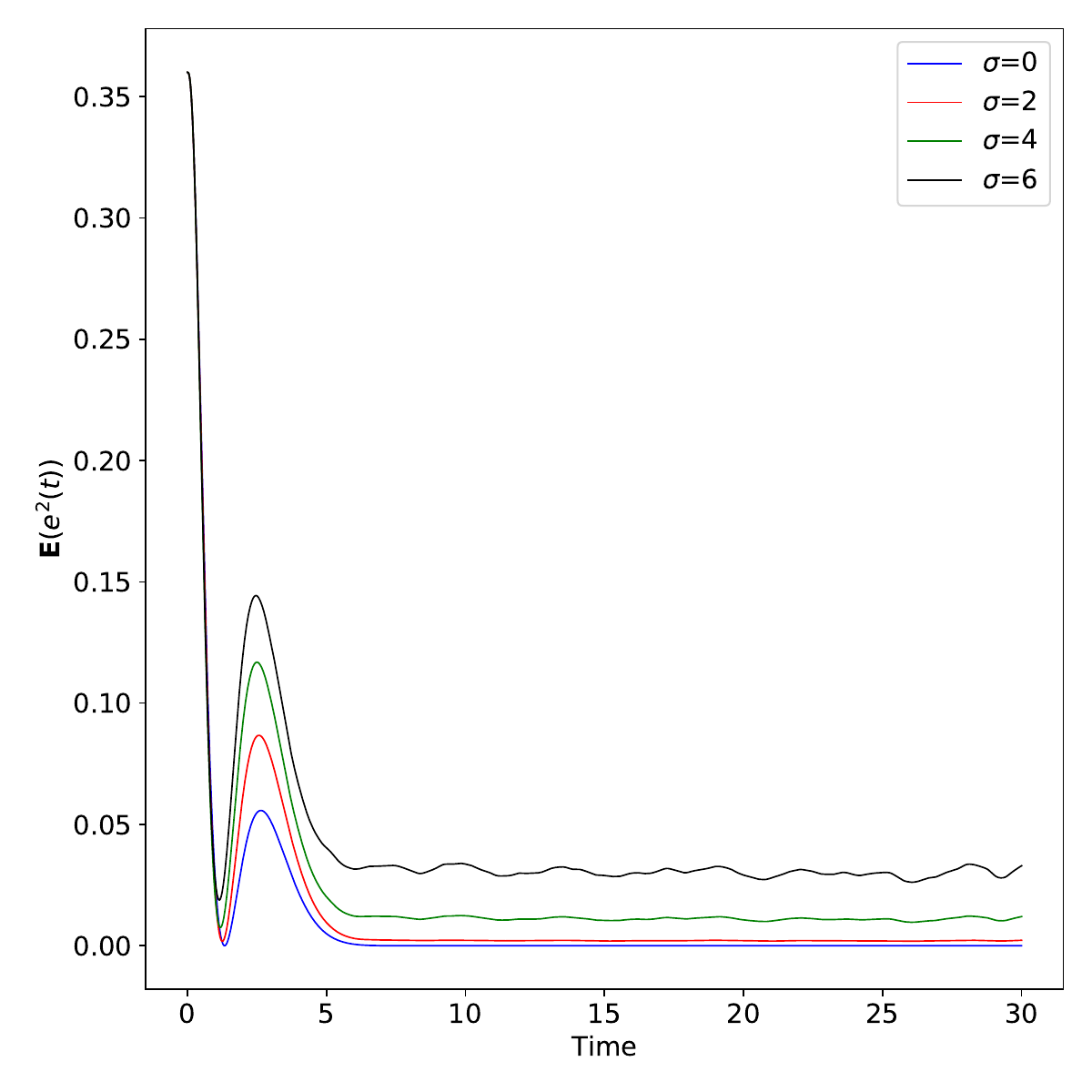}
 \caption{Curves of $\mathbf{E}|e(t)|^2$ under different  $\sigma$. The system parameters $(a,b,c,d,\mu)=(0.4,-0.3,0.5,6,5.2)$, and the initial state $(x_1(0),x_2(0),x_3(0))$ is $(0.9,0,0.1)$.}
 \label{fig2}
 \end{minipage}
 \end{figure}

 \begin{figure}[htbp]
 \begin{minipage}[t]{1\linewidth}
 \includegraphics[width=3.5in, height=3in]{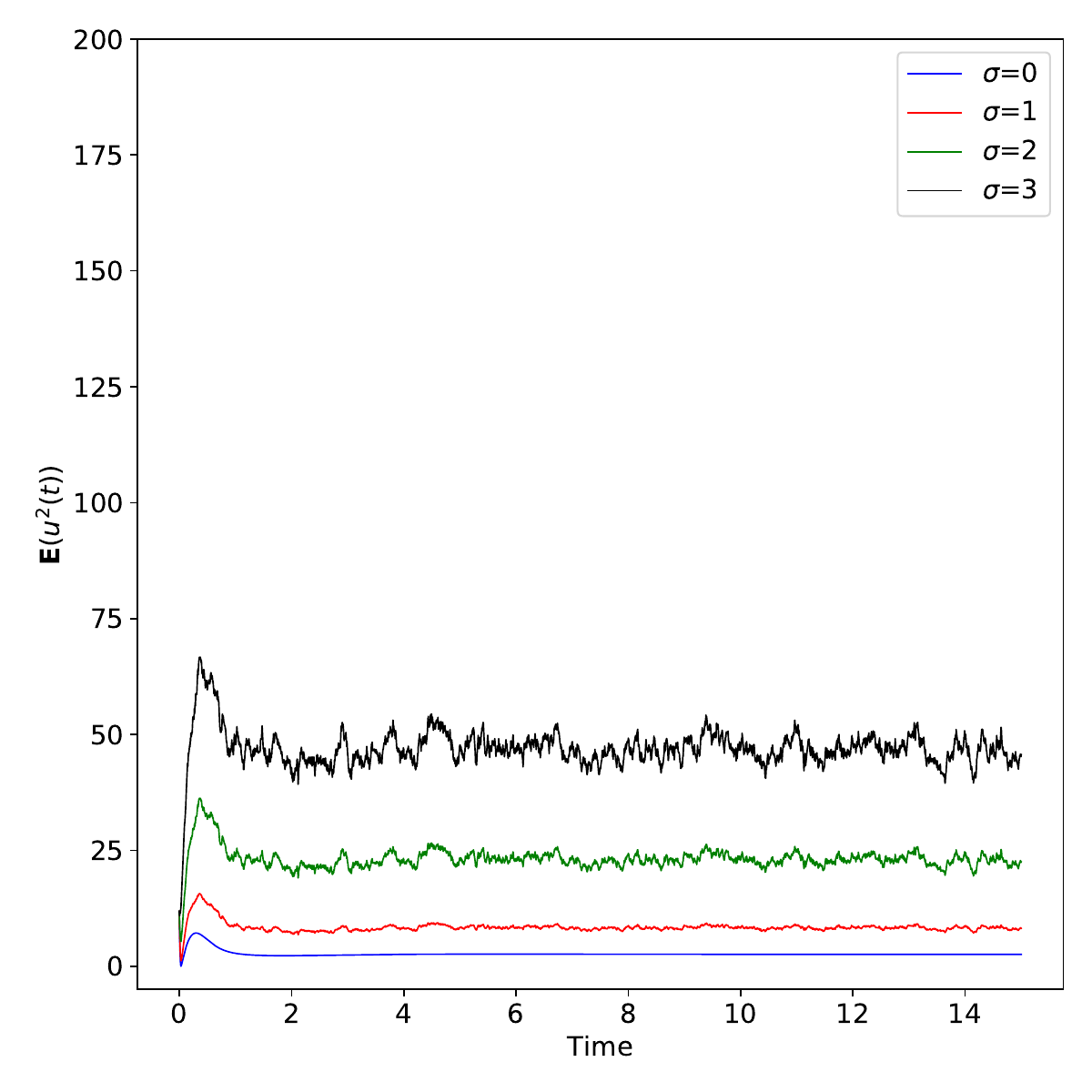}
 \end{minipage}
\begin{minipage}[t]{1\linewidth}
\includegraphics[width=3.5in, height=3in]{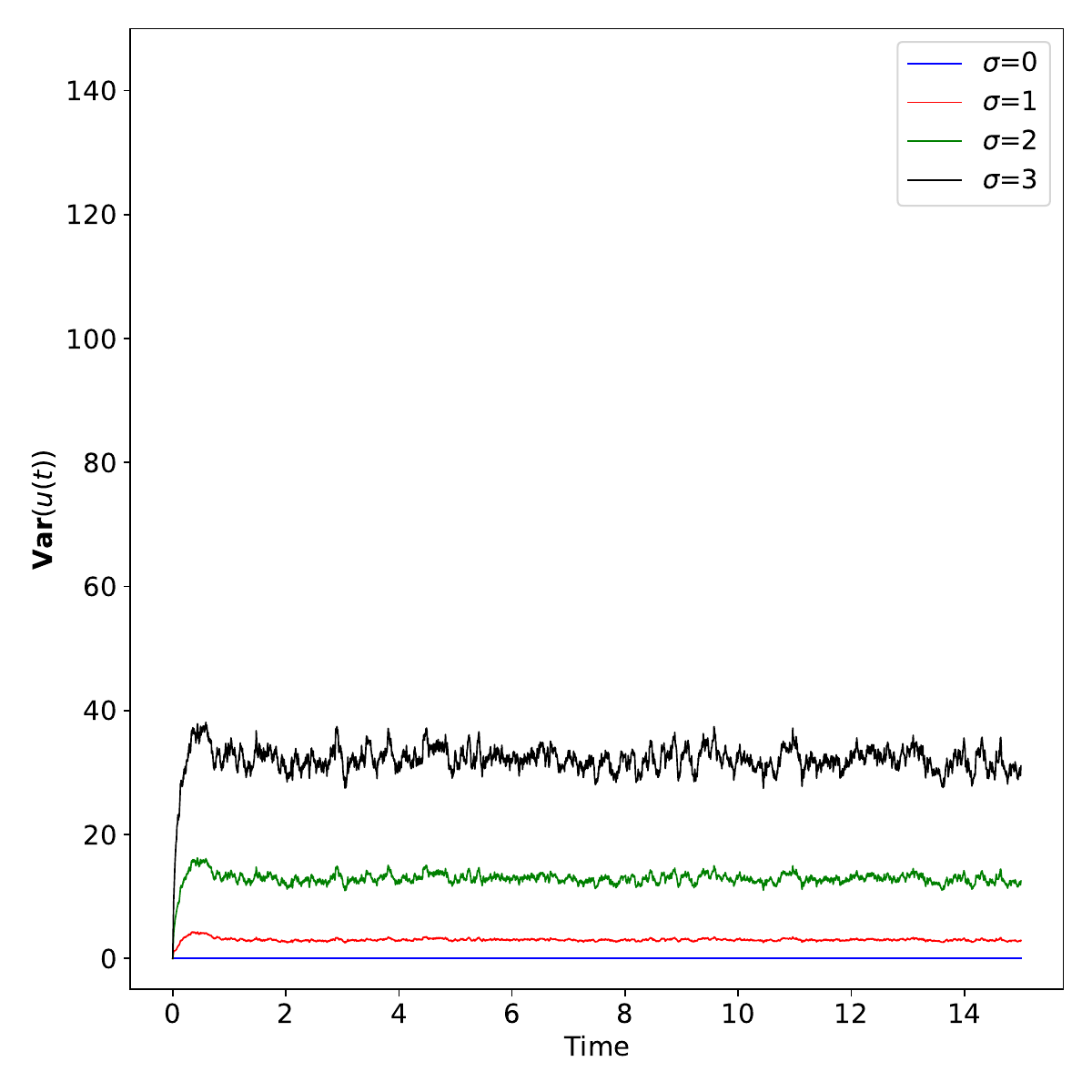}
\caption{Curves of $\mathbf{E}|u(t)|^2$ and $\text{Var}(u(t))$ under different $\sigma$. The system parameters $(a,b,c,d,\mu)=(0.4,-0.3,0.5,6,5.2)$, and the initial state $(x_1(0),x_2(0),x_3(0))$ is $(1.3,0,0.1)$.}
\label{fig3}
\end{minipage}
\end{figure}

\section{Conclusion}
This paper investigates the stabilization problem of the extended  PID control for a class of nonaffine stochastic systems with a general relative degree.  We have provided a mathematical theory together with explicit design formulae of the extended PID control for a basic class of nonlinear stochastic systems.  To be specific, an  $(n+1)$-dimensional unbounded set is constructed based on the Lipschitz constants of the drift and diffusion terms, from which the stabilizing extended PID parameters can be chosen arbitrarily.  Moreover, the steady-state tracking error  is proved to be bounded by the noise intensity at the setpoint, which can also be made arbitrarily small if one chooses the  extended PID parameters suitably large.  These theoretical results and design methods presented in the paper may provide reliable design guidance for control engineers in practical applications. For further investigation, it would be meaningful to consider more practical situations including time-delay and saturation, etc.

\section*{Appendix}
Consider the following stochastic differential equation:
\begin{equation}
	\begin{cases}
		\mathrm{d}x(t)& =b(t,x(t))\mathrm{d}t+\sigma(t,x(t))\mathrm{d}B(t),\\
		x(0)& =x_{0},
	\end{cases}\label{eq:sto}
\end{equation}
where $x(t)\in\mathbb{R}^{n}$ is the state, $B(t)\in\mathbb{R}^m$ is an $m$-dimensional
standard Brownian motion defined on a complete probability space $(\Omega,\mathcal{F},P)$
with $\{\mathcal{F}_{t}\}_{t\geq 0}$ being a natural filtration, and
$b\in C^{1}(\mathbb{R}^+\times\mathbb{R}^{n},\mathbb{R}^{n}),$ $\sigma\in C^{1}(\mathbb{R}^+\times\mathbb{R}^{n},\mathbb{R}^{n\times m})$ are nonlinear functions.

\begin{definition}\label{de1}
	Given a function $V\in C^{2}(\mathbb{R}^{n};\mathbb{R})$
	associated with the equation (\ref{eq:sto}). The differential
	operator $\mathcal{L}_t$ acting on $V$ is defined by
	\begin{align}\label{definition of operator L}
		\mathcal{L}_t V(x)=&\frac{\partial V}{\partial x}b(t,x)+\frac{1}{2}{\rm tr}\left\{\sigma(t,x)^{\top}\frac{\partial^{2}V}{\partial x^{2}}\sigma(t,x)\right\}.
	\end{align}
	If $b(t,x)$ and $\sigma(t,x)$ are independent of time $t$, the operator $\mathcal{L}_t$ will be denoted by $\mathcal{L}$ for simplicity.
\end{definition}

According to \cite[Theorem 3.5 and Remark 3.4]{khasminskii2012} up to a slight modification, we have the following lemma.
\begin{lemma}\label{lemma upper bound}
	If there exist a function $V\in C^{2}(\mathbb{R}^n;\mathbb{R}^+)$ and positive constants $\alpha,\beta$ such that $\liminf_{|x|\to \infty}V(x)=\infty$ and $\mathcal{L}_tV(x)\leq -\alpha V(x)+\beta$,
	then SDE \eqref{eq:sto} has a unique solution $(x(t))_{t\geq 0}$ such that
	\begin{equation*}
		\mathbf{E}V(x(t))\leq V(x(0))e^{-\alpha t}+\frac{\beta}{\alpha}.
	\end{equation*}
\vskip 0.2cm
\end{lemma}

Moreover, by a similar argument as in \cite[Theorem 3.5]{khasminskii2012}, we also have the following prior estimates.

\begin{lemma}\label{lemma lower bound}
	Suppose SDE \eqref{eq:sto} has a unique solution $(x(t))_{t\geq 0}$. If there exist a nonnegative function $V\in C^{2}(\mathbb{R}^{n};\mathbb{R})$, a constant $\alpha>0$ and an adapted process $\beta\in L_{loc}^1([0,\infty)\times \Omega;\mathbb{R})$ such that $\mathcal{L}_tV(x)\geq -\alpha V(x)+\beta(t),$
	then we have
	\begin{equation*}
		\mathbf{E}V(x(t))\geq V(x(0))e^{-\alpha t}+\int_{0}^{t}e^{-\alpha (t-s)}\mathbf{E}\beta(s)\mathrm{d}s,
	\end{equation*}
	where $L_{loc}^1([0,\infty)\times \Omega;\mathbb{R})$ is the space of stochastic processes $\beta(\cdot)$ such that $\beta(t)\in L^1(\Omega), \ t\geq 0$ and
	\begin{equation*}
		\mathbf{E}\int_0^T|\beta(t)|\mathrm{d}t<\infty, \ \text{ for all } \ T>0.
	\end{equation*}
\vskip 0.2cm
\end{lemma}

The real polynomial $f(s)=a_0+a_1s+\cdots+ a_ns^n$ is called stable if every one of its roots has negative part.
It is known that the stability of the positive-coefficients polynomial is determined by the $n-2$ determining coefficients $\alpha_i$, which are defined by
\[
\alpha_i:=\frac{a_{i-1}a_{i+2}}{a_ia_{i+1}}, ~i=1,\dots,n-2.\]
 The following lemma provides a sufficient condition for stability of the polynomial $f(s)$ for the case $n\geq 5$, see Theorem 5 in \cite{nie1987new}.
\begin{lemma} \label{lemma2}For $n\geq 5$, if every determining coefficients of $f(s)$ is less than $1/2$, and any three successive  determining coefficients satisfy the following condition
\[\alpha_i+\alpha_{i-1}\alpha_{i}\alpha_{i+1}\le \frac{1}{2},~~~i=2,\cdots,n-3\]
then the polynomial $f(s)$ is stable.
\end{lemma}
By Lemma \ref{lemma2}, we can obtain the following result.

\begin{lemma}\label{lemma hurwitz}
For $n\geq 4$, suppose that the parameters $k_0,\cdots,k_n$ are all positive and satisfy (\ref{solution}), then the matrix $A$ defined by (\ref{A}) is Hurwitz.
\end{lemma}
\begin{proof}
First, by definition (\ref{A}) of the matrix $A$, the characteristic polynomial of $A$ can be calculated as follows:
$$\det (sI_{n+1}-A)=s^{n+1}+k_ns^n+\cdots+k_1s+k_0,$$
which is a real polynomial with  positive-coefficients. Note that $n+1\geq 5$, we can apply Lemma \ref{lemma2} to analyze the stability of the characteristic polynomial of $A$.

First, note that the coefficient of $s^{n+1}$ is 1, it is not difficult to see the determining coefficients are
\begin{align*}
\alpha_i=\frac{k_{i-1}k_{i+2}}{k_ik_{i+1}}, ~i=1,\dots,n-2;~~\alpha_{n-1}=\frac{k_{n-2}}{k_{n-1}k_{n}}.
\end{align*}
Next, from $k_i^2>2k_{i-1}k_{i+1}$ and $k_{i+1}^2>2k_{i}k_{i+2}$, $i=1,\cdots,n-2$,  we obtain
$k_i^2k_{i+1}^2>4k_{i-1}k_ik_{i+1}k_{i+2}$, and thus
$$\alpha_i=\frac{k_{i-1}k_{i+2}}{k_ik_{i+1}}<\frac{1}{4},~1\le i\le n-2.$$
Besides, combine $k_{n-1}^2>2k_{n-2}k_{n}$ and $k_{n}^2>k_{n-1}$, we conclude that
$$\alpha_{n-1}=\frac{k_{n-2}}{k_{n-1}k_{n}}<\frac{1}{2}.$$
Therefore, every determining coefficients of $\det (sI_{n+1}-A)$ is less than $1/2$. In addition,
any three successive  determining coefficients satisfy
\[\alpha_i+\alpha_{i-1}\alpha_{i}\alpha_{i+1}\le \frac{1}{4}+\left(\frac{1}{4}\right)^2\frac{1}{2}=\frac{9}{32}<\frac{1}{2},~~~2\le i\le n-3,\]
which gives the Hurwitz property of $A$.
\end{proof}


\end{document}